\documentclass[11pt,a4paper]{article}

\usepackage[utf8]{inputenc}
\usepackage[T1]{fontenc}
\usepackage{lmodern}
\usepackage{amsmath,amssymb,amsthm}
\usepackage{mathrsfs}
\usepackage{mathtools}
\usepackage{bm}
\usepackage{graphicx}
\usepackage{hyperref}
\hypersetup{hidelinks}
\usepackage{enumitem}
\usepackage{geometry}
\usepackage{tikz}
\usetikzlibrary{arrows.meta}
\usepackage{float}
\usepackage{pgfplots}
\pgfplotsset{compat=1.16}

\newtheorem{theorem}{Theorem}[section]
\newtheorem{proposition}[theorem]{Proposition}
\newtheorem{lemma}[theorem]{Lemma}
\newtheorem{corollary}[theorem]{Corollary}

\newtheorem{definition}[theorem]{Definition}
\newtheorem{remark}[theorem]{Remark}

\newtheorem{openproblem}{Open Problem}[section]
\newtheorem{hypothesis}{Hypothesis}[section]

\title{\bfseries
	Geometry, Dynamics, and Topology of the Thickness Landscape:
	A Morse-Theoretic Analysis of the Return Map in the Class
	$\mathcal{O}_C$}

\author{
	M. El Morsalani
	\thanks{QWave Consult, Germany.
		\texttt{Mohamed.elmorsalani@qwave-consult.eu}}
	\\
	M. Barkatou
	\thanks{ISTM Laboratory, Chouaib Doukkali University, Morocco.
		\texttt{barkatou.m@ucd.ac.ma}}
}

\date{October 2026}

\begin{document}
\maketitle
	
\begin{abstract}
	We study the geometry, dynamics, and topology of the thickness
	landscape associated with the geometric return map in the class
	$\mathcal O_C$.  For a convex core $C$ and an admissible outer
	domain $\Omega$, the fixed points of the return map
	$F:\partial C\to\partial C$ are precisely the critical points of the
	thickness function $d$.  At such a point $c_*$ the exact
	linearization is
	\[
	DF(c_*)=I+B_*H_*,
	\]
	where $H_*=\operatorname{Hess}_{\partial C}d(c_*)$ and
	$B_*=d_*(I-d_*S_*)^{-1}$ is positive definite.
	
	We show that the self-adjoint preconditioned Hessian
	$K_*=B_*^{1/2}H_*B_*^{1/2}$ admits an exact representation in terms
	of the shape operators of the core and outer boundaries.  Hence
	curvature changes the return spectrum without changing the inertia
	of the thickness Hessian.  In the strict thickness-ascent regime,
	every negative eigenvalue $\mu$ of $K_*$ lies in $(-2,0)$, and for
	every Morse equilibrium
	\[
	\dim E^s(c_*)=\operatorname{ind}_M(c_*),
	\qquad
	\dim E^u(c_*)=m-\operatorname{ind}_M(c_*).
	\]
	
	Since $\partial C\cong\mathbb S^m$, the Morse inequalities become
	equilibrium-counting laws for the return dynamics; in particular,
	the number of equilibria is even.  We compute the local fixed-point
	index,
	\[
	\operatorname{ind}_{\mathrm{fix}}(F,c_*)
	=
	(-1)^{m+\operatorname{ind}_M(c_*)},
	\]
	derive the global Euler balance and $\deg(F)=1$, and identify the
	fixed-point index with the Euler-characteristic jump of the
	forward-invariant superlevel filtration.  We also discuss invariant
	manifolds and a conditional Morse--Smale structure, provide explicit
	geometric examples, and extend the spectral correspondence to
	Morse--Bott critical manifolds.  The resulting framework links
	boundary curvature, thickness geometry, discrete return dynamics,
	and Morse topology within a single geometrically generated system.
\end{abstract}



% =========================================================
% =========================================================
% =========================================================
\section{Introduction}
\label{sec:introduction}
% =========================================================

The return-map construction associated with a convex core
$C\subset\mathbb R^N$ and an outer domain
$\Omega\in\mathcal O_C$ produces a discrete dynamical system on the
boundary of the core.  The outer boundary is represented by the
normal graph
\[
\Phi(c)
=
c+d(c)\nu(c),
\qquad
c\in\partial C,
\]
where $d$ is the positive thickness function, and the return map
\[
F:\partial C\longrightarrow\partial C
\]
is obtained by following the inward normal of the outer boundary back
to its first intersection with the convex core.

Although the construction is geometric, its dynamics is governed to
a remarkable extent by the differential structure of the scalar
landscape $d$.

The geometric foundations of the class and of the return construction
were developed in
\cite{barkatou2002,barkatou2026return}.  In particular,
\[
F(c)=c
\quad\Longleftrightarrow\quad
\nabla_{\partial C}d(c)=0,
\]
so the fixed points of the return map are precisely the critical
points of the thickness function.  At such a point $c_*$,
\begin{equation}
DF(c_*)
=
I+B_*H_*,
\label{eq:intro-linearization}
\end{equation}
where
\[
H_*
=
\operatorname{Hess}_{\partial C}d(c_*),
\]
and
\[
B_*
=
d_*
\bigl(I-d_*S_*\bigr)^{-1},
\qquad
d_*=d(c_*).
\]
Here $S_*$ denotes the shape operator of the convex-core boundary in
the convention
\[
S=-D\nu|_{T\partial C}.
\]
The operator $B_*$ is self-adjoint and positive definite.  The
associated self-adjoint curvature-preconditioned Hessian
\begin{equation}
K_*
=
B_*^{1/2}H_*B_*^{1/2}
\label{eq:intro-K}
\end{equation}
is therefore the natural spectral object for the local return
dynamics.

A subsequent global analysis
\cite{barkatou2026global}
identified a strict thickness-ascent regime in which
\begin{equation}
d(F(c))-d(c)
\ge
\gamma
|\nabla_{\partial C}d(c)|^2,
\qquad
\gamma>0,
\label{eq:intro-ascent}
\end{equation}
away from the critical set.  In this regime
\[
\mathcal V(c)
=
d_{\max}-d(c)
\]
is a strict Lyapunov function relative to the fixed-point set.  Under
the corresponding isolation hypotheses, every orbit converges to a
critical point of $d$, while nontrivial periodic and recurrent motion
is excluded.

The local geometric mechanism and a global convergence regime are
therefore available.  The problem addressed here is the next one:
how is the differential and topological organization of the
thickness landscape transmitted to the exact discrete return
dynamics?

\paragraph{Relation to recent work.}
The question lies close to several developments at the interface of
Morse theory, topology, and dynamical systems.  The relation between
stable foliations and CW structures for Morse--Smale gradient-like
systems has been studied, for example, by Abbondandolo and
Majer~\cite{abbondandoloMajer2023}. Kirwan and
Penington~\cite{kirwanPenington2021} developed forms of Morse theory
that remain meaningful beyond classical nondegeneracy, while
Lefschetz-type phenomena for Morse--Smale discrete systems provide a
related fixed-point-theoretic perspective
\cite{graffLebiedzMyszkowski2019}.  More recently,
Mac\'ias-Virg\'os et al.~\cite{maciasVirgos2026} established
Morse--Bott inequalities for endomorphisms, illustrating the
continuing interaction between critical-set topology and discrete
dynamics.

The return map studied here belongs to a different and more
geometrically constrained setting. It is neither an abstract
gradient-like map nor a Morse--Smale system prescribed independently
on a manifold.  It is generated by the normal geometry of the pair
$(C,\Omega)$.  The objective is therefore to determine how the
Morse structure of the thickness landscape is transmitted through
the curvature of the two boundaries to the spectrum, stability,
fixed-point indices, invariant sets, and filtration topology of the
exact return map.

A first new geometric ingredient is that the thickness Hessian itself
admits an exact curvature representation.  At a critical point let
\[
L_*
=
I-d_*S_*,
\]
and let $S_*^\Omega$ denote the shape operator of the outer boundary
at
\[
x_*=\Phi(c_*).
\]
The normals of the two boundaries agree at a critical point, so their
tangent spaces can be identified. We prove
\begin{equation}
\boxed{
	H_*
	=
	L_*S_*^\Omega L_*
	-
	L_*S_*,
}
\label{eq:intro-H-curvature}
\end{equation}
and consequently
\begin{equation}
\boxed{
	K_*
	=
	d_*
	\left(
	L_*^{1/2}S_*^\Omega L_*^{1/2}
	-
	S_*
	\right).
}
\label{eq:intro-K-curvature}
\end{equation}
These identities require no alignment of the principal directions of
the two boundaries.

When the principal frames are aligned, the relation becomes
especially transparent. If
\[
\kappa_i^C,
\qquad
\kappa_i^\Omega
\]
are the corresponding usual principal curvatures and
\[
\kappa_i^{\mathrm{par}}
=
\frac{\kappa_i^C}
{1+d_*\kappa_i^C}
\]
is the curvature of the parallel surface of the core at distance
$d_*$, then the Morse index satisfies
\begin{equation}
\operatorname{ind}_M(c_*)
=
\#
\left\{
i:
\kappa_i^\Omega
>
\kappa_i^{\mathrm{par}}
\right\}.
\label{eq:intro-curvature-index}
\end{equation}
Thus the Morse type itself acquires a direct boundary-curvature
interpretation.

The central dynamical result follows from the positivity of $B_*$.
The congruence
\[
K_*
=
B_*^{1/2}H_*B_*^{1/2}
\]
preserves inertia.  Moreover, the strict-ascent property
\eqref{eq:intro-ascent} forces every negative eigenvalue $\mu$ of
$K_*$ to satisfy
\[
-2<\mu<0.
\]
Since the corresponding multiplier of the return map is
\[
\Lambda=1+\mu,
\]
every negative Morse direction becomes a stable spectral direction.

Consequently, if
\[
m
=
\dim\partial C
=
N-1,
\]
then every Morse equilibrium satisfies
\begin{equation}
\boxed{
	\dim E^s(c_*)
	=
	\operatorname{ind}_M(c_*),
	\qquad
	\dim E^u(c_*)
	=
	m-\operatorname{ind}_M(c_*).
}
\label{eq:intro-index-spectrum}
\end{equation}
In particular, Morse minima are repelling equilibria, Morse maxima
are attracting equilibria, and critical points of intermediate
index have saddle type.

The topology of the phase space makes this local correspondence
globally restrictive.  Since $C$ is a compact convex body with
nonempty interior,
\[
\partial C
\cong
\mathbb S^m.
\]
For a Morse thickness landscape,
\begin{equation}
M_d(t)
=
1+t^m+(1+t)Q(t),
\qquad
Q(t)\in\mathbb Z_{\ge0}[t].
\label{eq:intro-morse-polynomial}
\end{equation}
Writing
\[
Q(t)
=
\sum_{k=0}^{m-1}q_kt^k
\]
gives explicit relations among the numbers of critical points of
each Morse index and, in particular,
\begin{equation}
\boxed{
	\#\operatorname{Crit}(d)
	=
	2+2Q(1).
}
\label{eq:intro-even-count}
\end{equation}
Thus every Morse thickness landscape, and therefore every associated
return map, has an even number of equilibria.

Under strict ascent the same polynomial counts equilibria according
to stable spectral dimension. For example, on $\mathbb S^1$ the
numbers of repellers and attractors agree, while on $\mathbb S^2$
\[
\#\{\text{saddles}\}
=
\#\{\text{repellers}\}
+
\#\{\text{attractors}\}
-2.
\]

A further connection with topology is provided by the local
fixed-point index. Every Morse equilibrium is a nondegenerate fixed
point and
\begin{equation}
\boxed{
	\operatorname{ind}_{\mathrm{fix}}(F,c_*)
	=
	(-1)^{m+\operatorname{ind}_M(c_*)}.
}
\label{eq:intro-fixed-index}
\end{equation}
In the strict-ascent regime this can equivalently be written as
\[
\operatorname{ind}_{\mathrm{fix}}(F,c_*)
=
(-1)^{\dim E^u(c_*)}.
\]
Summing the local indices gives
\[
\sum_{c_*\in\operatorname{Fix}(F)}
\operatorname{ind}_{\mathrm{fix}}(F,c_*)
=
\chi(\partial C),
\]
and the Lefschetz--Hopf theorem then yields
\begin{equation}
\boxed{
	\deg(F)=1.
}
\label{eq:intro-degree}
\end{equation}
This conclusion does not require the strict-ascent hypothesis.

The level-set filtration of $d$ provides another manifestation of the
same structure. At a Morse critical point of index $\lambda$, the
sublevel filtration attaches a handle of index $\lambda$, whereas
the corresponding superlevel handle has index $m-\lambda$. Under
strict ascent,
\begin{equation}
\text{sublevel handle index}
=
\dim E^s(c_*),
\qquad
\text{superlevel handle index}
=
\dim E^u(c_*).
\label{eq:intro-filtration-spectrum}
\end{equation}
Moreover, the superlevel sets
\[
U^a
=
\{d\ge a\}
\]
are forward invariant, and the Euler-characteristic jump across an
isolated critical level is exactly the local fixed-point index of the
corresponding equilibrium. The thickness filtration is therefore
naturally oriented by the return dynamics.

The spectral correspondence also leads to invariant-manifold
statements. Local stable manifolds have dimension equal to the Morse
index. Local unstable manifolds have the complementary dimension
when the return map is locally invertible. Under additional global
invertibility and transversality hypotheses, the return dynamics is
Morse--Smale; its heteroclinic graph is acyclic and any connection
\[
p\longrightarrow q
\]
is constrained by
\[
d(p)<d(q),
\qquad
\operatorname{ind}_M(p)
\le
\operatorname{ind}_M(q).
\]
These conclusions are conditional where invertibility or
transversality is required and are kept separate from the intrinsic
consequences of the return geometry.

\paragraph{Why these results are of interest and potential applications.}
The significance of the preceding correspondence is that geometric
information contained in a scalar thickness landscape and in boundary
curvature can be converted into qualitative information about an
exact nonlinear discrete system without explicitly iterating that
system.  Critical points and curvature data determine local spectral
types and rates; Morse theory constrains the number and distribution
of equilibria; and the filtration records how the topology of
dynamically forward-invariant regions changes as thickness levels are
crossed.

This viewpoint is naturally related to geometric constructions
involving normal projections, parallel and offset surfaces, repeated
normal-return operations, and shape-dependent iterations. It may
therefore be useful in computational and differential geometry,
offset-surface analysis, normal-projection algorithms, and problems
in which a family of admissible domains is varied and one wishes to
track the creation, stability, or loss of equilibrium
configurations. The present paper does not pursue numerical
applications; its purpose is to establish the structural theory on
which such investigations can be based.

The analysis also distinguishes three spectral thresholds.  If
\[
\mu\in\sigma(K_*),
\qquad
\Lambda=1+\mu,
\]
then
\begin{equation}
\boxed{
	\begin{array}{c|c|c}
	\mu & \Lambda & \text{meaning}\\ \hline
	0   & 1  & \text{loss of Morse nondegeneracy},\\
	-1  & 0  & \text{loss of local invertibility},\\
	-2  & -1 & \text{loss of stability of a negative Morse direction}.
	\end{array}}
\label{eq:intro-three-thresholds}
\end{equation}
These phenomena are distinct. In particular, crossing $\mu=-2$
does not change the Morse index and, by itself, does not imply a
specific bifurcation.

Beyond the Morse regime, the same exact linearization and curvature
identities remain valid.  For Morse--Bott landscapes, a critical
submanifold $\Sigma$ becomes a fixed manifold and the strict-ascent
spectral splitting takes the form
\[
E^s
\oplus
T\Sigma
\oplus
E^u,
\]
with the transverse Morse--Bott index equal to $\dim E^s$. Under
appropriate invertibility assumptions these fixed manifolds are
normally hyperbolic. Genuinely degenerate landscapes possess
additional multiplier-one directions and require higher-order
analysis.

The main contributions of the paper may be summarized as follows.

\begin{enumerate}
	\item
	We derive the exact curvature identities
	\eqref{eq:intro-H-curvature}--\eqref{eq:intro-K-curvature}, valid
	without principal-direction alignment, and obtain a direct
	principal-curvature interpretation of the Morse index in the aligned
	case.
	
	\item
	We prove that the curvature-preconditioned Hessian preserves the
	inertia of the intrinsic thickness Hessian and that, under strict
	ascent, the Morse index equals the stable spectral dimension of the
	exact return map.
	
	\item
	We establish Morse genericity within admissible finite-dimensional
	families satisfying a parameter-spanning condition and translate the
	Morse inequalities on
	$\partial C\cong\mathbb S^m$ into equilibrium-counting relations,
	including the parity law
	\[
	\#\operatorname{Fix}(F)=2+2Q(1).
	\]
	
	\item
	We compute the local fixed-point index from the Morse index, derive
	the global Euler balance and degree-one property, and identify the
	local fixed-point index with the Euler-characteristic jump of the
	forward-invariant superlevel filtration.
	
	\item
	We relate Morse handle indices to stable and unstable spectral
	dimensions and develop the corresponding invariant-manifold
	consequences, including a conditional Morse--Smale description and
	the ordering of heteroclinic connections by thickness and Morse
	index.
	
	\item
	We give explicit circle, spherical, and noncommuting curvature
	examples and extend the local spectral and topological picture to
	Morse--Bott critical manifolds, including conditional normal
	hyperbolicity.
\end{enumerate}

The organization of the paper is as follows.
Section~\ref{sec:framework} recalls the geometric return-map
framework and the strict-ascent regime.
Section~\ref{sec:morse-landscapes} introduces Morse thickness
landscapes, the spherical topology of the convex-core boundary, and
genericity in admissible finite-dimensional families.
Section~\ref{sec:morse-spectrum} develops the curvature representation
and establishes the Morse-index/spectral-dimension correspondence.
Section~\ref{sec:morse-inequalities} derives the Morse and dynamical
equilibrium-counting relations.
Section~\ref{sec:euler-fixed-index} develops the fixed-point-index,
Euler-characteristic, and degree-one consequences.
Section~\ref{sec:thickness-filtration} relates the Morse filtration to
the stable and unstable spectral dimensions.
Section~\ref{sec:invariant-manifolds} treats invariant manifolds and
the conditional Morse--Smale structure.
Section~\ref{sec:examples} presents explicit examples illustrating
the different parts of the theory.
Section~\ref{sec:beyond-morse} treats Morse--Bott and genuinely
degenerate landscapes.
Section~\ref{sec:discussion-open-problems} summarizes the structural
picture and formulates the remaining open problems.
Finally, Section~\ref{sec:conclusion} concludes the paper.
% =========================================================
% =========================================================
\section{Geometric and Dynamical Framework}
\label{sec:framework}
% =========================================================

The purpose of this section is to collect the geometric and
dynamical ingredients from the preceding return-map theory that will
be used in the Morse-theoretic analysis
\cite{barkatou2002,barkatou2026return,barkatou2026global}.
No new local approximation or global convergence theorem is proved
here. Throughout the paper we use the exact geometric return map and
the curvature-preconditioned ascent structure established in
\cite{barkatou2026return,barkatou2026global}.

% ---------------------------------------------------------
\subsection{The class $\mathcal O_C$ and the return map}
\label{subsec:OC-return-map}
% ---------------------------------------------------------

Let
\[
C\subset\mathbb R^N,
\qquad
N\ge2,
\]
be a compact convex body with nonempty interior. We assume
throughout that
\[
\partial C\in C^3.
\]
The outward unit normal to $\partial C$ is denoted by
\[
\nu:\partial C\longrightarrow\mathbb S^{N-1}.
\]

Let $\Omega\in\mathcal O_C$. We do not repeat here the complete
definition of the class $\mathcal O_C$; see
\cite{barkatou2002,barkatou2026return}. We use the consequences of
the $C$-geometric normal property, regularity, radial accessibility,
and transversality that are relevant to the present analysis.
In particular, every outward normal half-line
\[
\Delta_c^+
=
\{c+r\nu(c):r\ge0\},
\qquad
c\in\partial C,
\]
meets the outer boundary in a unique transverse point. Consequently,
there exists a positive thickness function
\[
d:\partial C\longrightarrow(0,\infty)
\]
such that the outer boundary is represented by the normal graph
\begin{equation}
\Phi(c)
=
c+d(c)\nu(c).
\label{eq:radial-parametrization}
\end{equation}

For the analysis below we assume
\[
d\in C^2(\partial C).
\]

Starting from
\[
x=\Phi(c),
\]
one follows the inward unit normal $n(x)$ of the outer boundary to
its first intersection with the convex core. Denoting the resulting
first-intersection map by
\[
\pi:\partial\Omega\setminus C\longrightarrow\partial C,
\]
the return map is
\begin{equation}
F
=
\pi\circ\Phi:
\partial C\longrightarrow\partial C.
\label{eq:return-map}
\end{equation}
The associated discrete dynamics is
\begin{equation}
c_{k+1}=F(c_k).
\label{eq:return-dynamics}
\end{equation}

The global results recalled below require the uniform
return-transversality and small-slope hypotheses introduced in
\cite{barkatou2026global}. The Morse-theoretic statements concerning
the thickness function alone do not depend on these additional
dynamical assumptions.

% ---------------------------------------------------------
\subsection{Thickness, curvature, and the preconditioning operator}
\label{subsec:preconditioning}
% ---------------------------------------------------------

We use the sign convention
\begin{equation}
S_c
=
-D\nu(c)|_{T_c\partial C}
\label{eq:shape-operator}
\end{equation}
for the shape operator of the convex core. If
\[
\kappa_1(c),\ldots,\kappa_{N-1}(c)\ge0
\]
are the usual principal curvatures of $\partial C$, then the
eigenvalues of $S_c$ are
\[
-\kappa_1(c),\ldots,-\kappa_{N-1}(c).
\]
Hence
\[
I-d(c)S_c
\]
is positive definite.

Define
\begin{equation}
B_c
=
d(c)\bigl(I-d(c)S_c\bigr)^{-1}.
\label{eq:preconditioner}
\end{equation}
Then $B_c$ is self-adjoint and positive definite on
$T_c\partial C$. In a principal direction $e_i(c)$,
\begin{equation}
B_ce_i(c)
=
\frac{d(c)}
{1+d(c)\kappa_i(c)}
e_i(c).
\label{eq:B-principal}
\end{equation}

Let
\[
g_c=\nabla_{\partial C}d(c),
\]
and let
\[
H_c:T_c\partial C\longrightarrow T_c\partial C
\]
be the self-adjoint operator associated with the intrinsic Hessian:
\begin{equation}
\langle H_cv,w\rangle
=
\operatorname{Hess}_{\partial C}d(c)[v,w].
\label{eq:Hessian-operator}
\end{equation}

The corresponding curvature-preconditioned Hessian is
\begin{equation}
K_c
=
B_c^{1/2}H_cB_c^{1/2}.
\label{eq:preconditioned-Hessian}
\end{equation}
Unlike $B_cH_c$, the operator $K_c$ is self-adjoint. Since
\begin{equation}
B_cH_c
=
B_c^{1/2}K_cB_c^{-1/2},
\label{eq:similarity-BH-K}
\end{equation}
the two operators have the same real spectrum.

% ---------------------------------------------------------
\subsection{Fixed points and exact linearization}
\label{subsec:fixed-linearization}
% ---------------------------------------------------------

The foundational return-map theory identifies equilibria with
critical points of the thickness function:
\begin{equation}
F(c)=c
\quad\Longleftrightarrow\quad
\nabla_{\partial C}d(c)=0.
\label{eq:fix-crit}
\end{equation}
Thus
\begin{equation}
\operatorname{Fix}(F)
=
\operatorname{Crit}(d).
\label{eq:fix-equals-crit}
\end{equation}

Let
\[
c_*\in\operatorname{Crit}(d)
\]
and write
\[
B_*=B_{c_*},
\qquad
H_*=H_{c_*},
\qquad
K_*=B_*^{1/2}H_*B_*^{1/2}.
\]
At such a point, the derivative of the exact return map is
\begin{equation}
DF(c_*)
=
I+B_*H_*.
\label{eq:exact-linearization}
\end{equation}

If
\[
\mu_1,\ldots,\mu_{N-1}
\]
are the eigenvalues of $K_*$, then
\eqref{eq:similarity-BH-K} implies that the multipliers of the
return map are
\begin{equation}
\Lambda_j
=
1+\mu_j,
\qquad
j=1,\ldots,N-1.
\label{eq:return-multipliers}
\end{equation}

The relation between the Morse index of $d$ and the spectrum of
$DF(c_*)$ is the central subject of
Section~\ref{sec:morse-spectrum}.

% ---------------------------------------------------------
\subsection{The global ascent regime}
\label{subsec:global-ascent}
% ---------------------------------------------------------

The global theory developed in
\cite{barkatou2026global} gives, under uniform
return-transversality, a sufficiently small global slope, and a
uniform preconditioned-ascent condition, an estimate of the form
\begin{equation}
d(F(c))-d(c)
\ge
\gamma
|\nabla_{\partial C}d(c)|^2,
\qquad
\gamma>0.
\label{eq:global-ascent}
\end{equation}
Consequently,
\begin{equation}
\mathcal V(c)
=
d_{\max}-d(c)
\label{eq:Lyapunov}
\end{equation}
is a strict Lyapunov function relative to the fixed-point set.

The second-order mechanism underlying
\eqref{eq:global-ascent} is governed by the dynamically sampled
direction
\begin{equation}
z_c
=
B_c^{1/2}\nabla_{\partial C}d(c).
\label{eq:sampled-direction}
\end{equation}
More precisely, the relevant quadratic form is
\[
\left\langle
z_c,
\left(I+\frac12K_c\right)z_c
\right\rangle.
\]
Thus a uniform directional lower bound of the form
\begin{equation}
\left\langle
z_c,
\left(I+\frac12K_c\right)z_c
\right\rangle
\ge
\alpha_0|z_c|^2,
\qquad
\alpha_0>0,
\label{eq:directional-ascent}
\end{equation}
is the natural ascent condition.

A convenient stronger sufficient criterion is
\begin{equation}
\lambda_{\min}(K_c)
\ge
-2+\delta
\qquad
\text{for every }c\in\partial C,
\label{eq:spectral-ascent}
\end{equation}
for some $\delta>0$, together with the required uniform
return-transversality and small-slope assumptions. The spectral
condition controls all tangent directions, including directions that
may never be sampled by the return dynamics, and is therefore in
general stronger than the directional condition
\eqref{eq:directional-ascent}.

\begin{remark}[Role of the small-slope hypothesis]
	\label{rem:small-slope-role}
	The small-slope condition used in the global theory is a sufficient
	geometric hypothesis ensuring that the local return estimates can be
	made uniform over $\partial C$. No necessity is claimed. The exact
	return map may remain well defined outside this perturbative regime,
	but the uniform ascent estimate
	\eqref{eq:global-ascent} is then not asserted by the present theory.
\end{remark}

Under the global ascent hypotheses and the assumption that the
critical points are isolated, every orbit converges to a unique
critical point of $d$. Moreover, nontrivial periodic and recurrent
orbits are excluded. In the present paper the Morse hypothesis will
imply, in particular, that the critical set is finite and therefore
isolated.

The Morse assumption thus adds a further layer of structure: it not
only discretizes the equilibrium set but assigns an index to every
equilibrium. Because $B_*$ is positive definite, this index will be
shown in Section~\ref{sec:morse-spectrum} to determine exactly the
dimensions of the stable and unstable spectral subspaces of the
return map.
% =========================================================
% =========================================================
\section{Morse Thickness Landscapes}
\label{sec:morse-landscapes}
% =========================================================

The thickness function
\[
d:\partial C\longrightarrow(0,\infty)
\]
will now be regarded as a $C^2$ landscape on the compact manifold
$\partial C$.  The purpose of this section is to introduce the Morse
structure of this landscape independently of the dynamical
assumptions used in Section~\ref{sec:framework}.  In particular, the
Morse-theoretic statements below do not require the global
small-slope or ascent hypotheses.

Set
\begin{equation}
m:=\dim\partial C=N-1.
\label{eq:phase-dimension}
\end{equation}

% ---------------------------------------------------------
\subsection{Morse functions on $\partial C$}
\label{subsec:morse-functions}
% ---------------------------------------------------------

We recall the standard notion of a Morse function on a compact
manifold; see, for example,
\cite{milnor1963morse,bott1982morse,nicolaescu2011morse}.

\begin{definition}[Morse thickness landscape]
	\label{def:morse-thickness}
	The thickness function
	\[
	d\in C^2(\partial C)
	\]
	is called a \emph{Morse thickness function} if every critical point
	$c_*\in\operatorname{Crit}(d)$ is nondegenerate, that is,
	\begin{equation}
	\det H_*\neq0,
	\qquad
	H_*=\operatorname{Hess}_{\partial C}d(c_*).
	\label{eq:morse-nondegeneracy}
	\end{equation}
	The pair $(\partial C,d)$ will then be called a
	\emph{Morse thickness landscape}.
\end{definition}

Nondegeneracy is coordinate invariant.  Equivalently, at a Morse
critical point,
\[
H_*:T_{c_*}\partial C\longrightarrow T_{c_*}\partial C
\]
has no zero eigenvalue.

\begin{proposition}[Finiteness of the critical set]
	\label{prop:finite-critical-set}
	If $d$ is a Morse thickness function, then
	$\operatorname{Crit}(d)$ is finite. Consequently,
	\begin{equation}
	\operatorname{Fix}(F)
	=
	\operatorname{Crit}(d)
	\label{eq:finite-fix-crit}
	\end{equation}
	is a finite set.
\end{proposition}

\begin{proof}
	A nondegenerate critical point is isolated by the inverse function
	theorem applied in local coordinates to the gradient field.
	Moreover, $\operatorname{Crit}(d)$ is closed because
	$\nabla_{\partial C}d$ is continuous.  Hence
	$\operatorname{Crit}(d)$ is a closed discrete subset of the compact
	manifold $\partial C$, and therefore is finite.  The second assertion
	follows from \eqref{eq:fix-equals-crit}.
\end{proof}

\begin{remark}[Classical versus admissible genericity]
	\label{rem:classical-admissible-genericity}
	On a compact smooth manifold, Morse functions form an open dense
	subset of the appropriate $C^2$ function space
	\cite{milnor1963morse,nicolaescu2011morse}.  This classical
	genericity statement should not be confused with genericity in the
	geometrically constrained class $\mathcal O_C$.  An arbitrary
	perturbation of $d$ need not generate an admissible outer boundary
	satisfying the geometric normal and return-transversality conditions.
	A conditional genericity statement for admissible finite-dimensional
	families is given in Subsection~\ref{subsec:admissible-genericity}.
\end{remark}

We do not assume that distinct critical points have distinct critical
values. Whenever a filtration argument requires separation of
critical levels, the relevant interval will be chosen to isolate the
critical value under consideration, or the standard Morse-theoretic
formulation for several critical points at the same level will be
used.

% ---------------------------------------------------------
\subsection{Morse index and critical-point types}
\label{subsec:morse-index}
% ---------------------------------------------------------

\begin{definition}[Morse index]
	\label{def:morse-index}
	Let $c_*\in\operatorname{Crit}(d)$ be a Morse critical point.
	Its Morse index is
	\begin{equation}
	\operatorname{ind}_M(c_*)
	:=
	n_-(H_*),
	\label{eq:morse-index}
	\end{equation}
	where $n_-(H_*)$ denotes the number of negative eigenvalues of
	$H_*$, counted with multiplicity.
\end{definition}

Since $H_*$ is nonsingular and acts on the $m$-dimensional tangent
space $T_{c_*}\partial C$,
\begin{equation}
n_-(H_*)+n_+(H_*)=m.
\label{eq:inertia-dimension}
\end{equation}

The Morse lemma gives the local geometric meaning of the index.

\begin{theorem}[Morse normal form]
	\label{thm:morse-normal-form}
	Let $c_*$ be a critical point of a Morse thickness function and set
	\[
	\lambda=\operatorname{ind}_M(c_*).
	\]
	Then there exist local coordinates
	\[
	x=(x_1,\ldots,x_m)
	\]
	centered at $c_*$ in which
	\begin{equation}
	d(x)
	=
	d(c_*)
	-
	\sum_{j=1}^{\lambda}x_j^2
	+
	\sum_{j=\lambda+1}^{m}x_j^2.
	\label{eq:morse-normal-form}
	\end{equation}
\end{theorem}

The normalization of the quadratic coefficients in
\eqref{eq:morse-normal-form} is immaterial for the return dynamics;
what will matter is the splitting into negative and positive
directions.

\begin{corollary}[Critical-point types]
	\label{cor:morse-critical-types}
	Let $c_*$ be a Morse critical point of $d$. Then:
	\begin{enumerate}
		\item $c_*$ is a strict local minimum if and only if
		\[
		\operatorname{ind}_M(c_*)=0;
		\]
		\item $c_*$ is a strict local maximum if and only if
		\[
		\operatorname{ind}_M(c_*)=m;
		\]
		\item $c_*$ is a saddle point if and only if
		\[
		0<\operatorname{ind}_M(c_*)<m.
		\]
	\end{enumerate}
\end{corollary}

\begin{remark}[Morse type versus dynamical type]
	\label{rem:morse-versus-dynamical}
	Corollary~\ref{cor:morse-critical-types} concerns the thickness
	landscape alone.  It does not yet identify minima, maxima, or saddles
	with repelling, attracting, or mixed fixed points of the return map.
	That correspondence requires the exact linearization
	\[
	DF(c_*)=I+B_*H_*
	\]
	together with the strict-ascent information.  It will be established
	in Section~\ref{sec:morse-spectrum}.
\end{remark}

% ---------------------------------------------------------
\subsection{Topology of the convex-core boundary}
\label{subsec:core-boundary-topology}
% ---------------------------------------------------------

The topology of the phase space is particularly rigid in the present
setting. Since $C$ is a compact convex body with nonempty interior,
its boundary cannot have arbitrary topology.

\begin{proposition}[Topology of the convex-core boundary]
	\label{prop:core-boundary-sphere}
	Let
	\[
	C\subset\mathbb R^N
	\]
	be a compact convex body with nonempty interior and
	$\partial C\in C^3$. Then
	\begin{equation}
	\partial C\cong\mathbb S^{N-1}
	\label{eq:core-boundary-sphere}
	\end{equation}
	diffeomorphically.
\end{proposition}

\begin{proof}
	Choose
	\[
	a\in\operatorname{int}(C).
	\]
	For every $\omega\in\mathbb S^{N-1}$, the ray
	\[
	\{a+r\omega:r\ge0\}
	\]
	meets $\partial C$ in exactly one point, by convexity and boundedness
	of $C$. Thus there exists a unique radial function
	\[
	\rho:\mathbb S^{N-1}\longrightarrow(0,\infty)
	\]
	such that
	\[
	a+\rho(\omega)\omega\in\partial C.
	\]
	Define
	\begin{equation}
	\Psi(\omega)
	=
	a+\rho(\omega)\omega.
	\label{eq:radial-diffeomorphism}
	\end{equation}
	The map $\Psi$ is bijective, with inverse
	\begin{equation}
	\Psi^{-1}(c)
	=
	\frac{c-a}{|c-a|},
	\qquad
	c\in\partial C.
	\label{eq:radial-diffeomorphism-inverse}
	\end{equation}
	
	At every $c\in\partial C$, convexity and
	$a\in\operatorname{int}(C)$ imply
	\begin{equation}
	\langle c-a,\nu(c)\rangle>0.
	\label{eq:radial-transversality}
	\end{equation}
	Hence the radial direction is transverse to $\partial C$.
	The implicit function theorem therefore gives $\rho$ the same
	regularity as $\partial C$, and $\Psi$ is a diffeomorphism.
	See also the standard radial description of convex bodies in
	\cite{schneider2014convex}.
\end{proof}

\begin{remark}
	\label{rem:no-strict-convexity-topology}
	No strict-convexity or positive lower-curvature assumption is needed
	for Proposition~\ref{prop:core-boundary-sphere}.  Convexity,
	nonempty interior, and the stated boundary regularity are sufficient.
\end{remark}

Thus the homology of the phase space is already fixed by convexity.
For coefficients in any field $\mathbb F$,
\begin{equation}
H_k(\partial C;\mathbb F)
\cong
\begin{cases}
\mathbb F, & k=0,m,\\
0, & 0<k<m.
\end{cases}
\label{eq:sphere-homology}
\end{equation}
Consequently,
\begin{equation}
b_k(\partial C)
=
\begin{cases}
1, & k=0,m,\\
0, & 0<k<m,
\end{cases}
\label{eq:sphere-betti}
\end{equation}
and
\begin{equation}
\chi(\partial C)
=
1+(-1)^m
=
1+(-1)^{N-1}.
\label{eq:core-euler-characteristic}
\end{equation}

Equivalently,
\[
\chi(\partial C)
=
\begin{cases}
2, & N\ \text{odd},\\
0, & N\ \text{even}.
\end{cases}
\]

\begin{remark}[Why the topology is still nontrivial]
	\label{rem:landscape-not-phase-topology}
	Although the topology of $\partial C$ is fixed, the Morse theory of
	the thickness landscape is far from trivial.  The relevant
	topological information lies in the evolution of the sublevel and
	superlevel sets of $d$ as critical levels are crossed.  The Morse
	indices determine the handle dimensions, while
	Section~\ref{sec:morse-spectrum} will show that the same indices
	determine the stable and unstable spectral dimensions of the return
	map.  The variable topology is therefore the topology of the
	\emph{thickness filtration} on a spherical phase space, rather than
	the topology of the phase space itself.
\end{remark}

From \eqref{eq:sphere-betti}, the Poincar\'e polynomial is
\begin{equation}
P_{\partial C}(t)
=
1+t^m.
\label{eq:sphere-poincare}
\end{equation}
Section~\ref{sec:morse-inequalities} will compare this polynomial
with the Morse polynomial of $d$ and translate the resulting
relations into constraints on the equilibria of the return map.

% ---------------------------------------------------------
\subsection{Genericity in admissible finite-dimensional families}
\label{subsec:admissible-genericity}
% ---------------------------------------------------------

The classical genericity of Morse functions does not by itself imply
genericity inside $\mathcal O_C$.  Nevertheless, a useful conditional
statement is available for finite-dimensional families that are
known a priori to remain admissible.

Let
\[
A\subset\mathbb R^K
\]
be open, and let
\[
d_a
=
d_0+\sum_{j=1}^{K}a_j\varphi_j,
\qquad
a=(a_1,\ldots,a_K)\in A,
\]
where
\[
\varphi_j\in C^2(\partial C).
\]
Assume that every $d_a$ is positive and generates an admissible
domain in $\mathcal O_C$ with the regularity required for the return
construction.

\begin{hypothesis}[Parameter spanning]
	\label{hyp:parameter-spanning}
	For every $c\in\partial C$, the covectors
	\[
	d\varphi_1(c),\ldots,d\varphi_K(c)
	\]
	span $T_c^*\partial C$.
\end{hypothesis}

\begin{theorem}[Genericity in an admissible parameter family]
	\label{thm:admissible-morse-genericity}
	Under Hypothesis~\ref{hyp:parameter-spanning}, the set
	\[
	A_{\mathrm{Morse}}
	:=
	\{a\in A:d_a\ \text{is Morse}\}
	\]
	is residual in $A$.  Since the Morse property is open in the
	$C^2$ topology, $A_{\mathrm{Morse}}$ is in fact open and dense in
	$A$.
\end{theorem}

\begin{proof}
	Consider the $C^1$ map
	\[
	G:A\times\partial C\longrightarrow T^*\partial C,
	\qquad
	G(a,c)=d_{\partial C}d_a(c),
	\]
	and let $Z\subset T^*\partial C$ denote the zero section.
	
	At a point $(a,c)$ with $G(a,c)=0$, variation in the parameter
	direction $v=(v_1,\ldots,v_K)\in\mathbb R^K$ gives the vertical
	component
	\[
	D_aG(a,c)[v]
	=
	\sum_{j=1}^{K}v_j\,d\varphi_j(c).
	\]
	By Hypothesis~\ref{hyp:parameter-spanning}, these parameter
	variations span the entire vertical fiber $T_c^*\partial C$.
	Consequently,
	\[
	G\pitchfork Z.
	\]
	
	The parametric transversality theorem \cite{hirsch1976} therefore implies that, for a
	residual set of parameters $a\in A$, the section
	\[
	c\longmapsto d_{\partial C}d_a(c)
	\]
	is transverse to the zero section.  At a critical point of $d_a$,
	this transversality is equivalent to nonsingularity of the Hessian.
	Hence $d_a$ is Morse for a residual set of parameters.
	
	Finally, the map
	\[
	a\longmapsto d_a
	\]
	is continuous from $A$ into $C^2(\partial C)$, and the set of Morse
	functions is open in $C^2(\partial C)$. Therefore
	$A_{\mathrm{Morse}}$ is open. Since it contains a residual subset of
	the Baire space $A$, it is dense.
\end{proof}

\begin{remark}[Scope of the genericity theorem]
	\label{rem:genericity-scope}
	Theorem~\ref{thm:admissible-morse-genericity} does not assert that
	Morse thickness functions are generic in the full class
	$\mathcal O_C$.  Admissibility of the family $\{d_a\}_{a\in A}$ is
	an explicit hypothesis.  The theorem states that, once an admissible
	finite-dimensional family is available and its parameter directions
	span the cotangent spaces, Morse nondegeneracy is generic within that
	family.  Genericity in the full geometrically constrained class
	remains a separate problem.
\end{remark}
% =========================================================
% =========================================================
\section{Morse Index and the Spectrum of the Return Map}
\label{sec:morse-spectrum}
% =========================================================

We now connect the Morse index of the thickness landscape with the
linearized dynamics of the return map.  The key observation is that
the curvature preconditioner changes the scale and anisotropy of the
thickness Hessian but not its inertia.  Combined with strict
thickness ascent, this yields an exact correspondence between Morse
index and stable spectral dimension.

A second objective of this section is to make the curvature content
of this correspondence explicit.  At a critical point, the
thickness Hessian can be expressed directly in terms of the shape
operators of the convex core and the outer boundary.

Throughout this section let
\[
c_*\in\operatorname{Crit}(d)
\]
be a Morse critical point.  Write
\[
d_*=d(c_*),
\qquad
H_*=\operatorname{Hess}_{\partial C}d(c_*),
\]
\[
L_*=I-d_*S_*,
\qquad
B_*=d_*L_*^{-1},
\]
and
\[
K_*=B_*^{1/2}H_*B_*^{1/2}.
\]
Recall from Section~\ref{sec:framework} that
\begin{equation}
DF(c_*)=I+B_*H_*.
\label{eq:section4-linearization}
\end{equation}

% ---------------------------------------------------------
\subsection{Inertia of the preconditioned Hessian}
\label{subsec:inertia-preconditioned}
% ---------------------------------------------------------

We first record the algebraic fact that makes the Morse index
compatible with the curvature-preconditioned dynamics.

\begin{proposition}[Inertia preservation]
	\label{prop:inertia-preservation}
	At every Morse critical point $c_*$,
	\begin{equation}
	\operatorname{inertia}(K_*)
	=
	\operatorname{inertia}(H_*).
	\label{eq:inertia-preservation}
	\end{equation}
	In particular,
	\begin{equation}
	n_-(K_*)
	=
	\operatorname{ind}_M(c_*),
	\label{eq:negative-inertia}
	\end{equation}
	and
	\begin{equation}
	n_+(K_*)
	=
	m-\operatorname{ind}_M(c_*).
	\label{eq:positive-inertia}
	\end{equation}
	Moreover, $K_*$ is nonsingular.
\end{proposition}

\begin{proof}
	Since $B_*$ is positive definite,
	$B_*^{1/2}$ is an invertible self-adjoint operator.  The quadratic
	forms associated with $H_*$ and $K_*$ satisfy
	\[
	\langle K_*v,v\rangle
	=
	\left\langle
	H_*B_*^{1/2}v,
	B_*^{1/2}v
	\right\rangle.
	\]
	Thus $K_*$ is obtained from $H_*$ by an invertible congruence
	transformation.  Sylvester's law of inertia gives
	\eqref{eq:inertia-preservation}.
	
	Since $c_*$ is Morse, $H_*$ is nonsingular, and therefore so is
	$K_*$.  Equations~\eqref{eq:negative-inertia} and
	\eqref{eq:positive-inertia} follow.
\end{proof}

\begin{remark}[Curvature changes rates, not Morse type]
	\label{rem:curvature-rates}
	The operator $B_*$ depends explicitly on the thickness and on the
	curvature of the convex core.  It can therefore substantially alter
	the magnitudes and eigendirections relevant to the return
	linearization.  Proposition~\ref{prop:inertia-preservation} shows,
	however, that this curvature preconditioning cannot change the
	number of positive and negative Hessian directions.
\end{remark}

The operator entering the derivative is $B_*H_*$ rather than $K_*$.
They nevertheless have exactly the same spectrum because
\begin{equation}
B_*H_*
=
B_*^{1/2}K_*B_*^{-1/2}.
\label{eq:BH-K-similarity}
\end{equation}
Consequently, if
\[
\mu_1,\ldots,\mu_m
\]
are the eigenvalues of $K_*$, then the multipliers of the return map
are
\begin{equation}
\Lambda_j
=
1+\mu_j,
\qquad
j=1,\ldots,m.
\label{eq:section4-multipliers}
\end{equation}
Because $K_*$ is self-adjoint, all $\mu_j$ are real, and
$B_*H_*$ is diagonalizable over $\mathbb R$.

% ---------------------------------------------------------
\subsection{Curvature representation at a critical point}
\label{subsec:curvature-representation}
% ---------------------------------------------------------

The operator $H_*$ admits a direct geometric interpretation in terms
of the curvature of the two boundaries.

Set
\[
x_*=\Phi(c_*)\in\partial\Omega.
\]
Since
\[
\nabla_{\partial C}d(c_*)=0,
\]
the outer unit normal of $\partial\Omega$ at $x_*$ agrees with the
outer unit normal of the core:
\[
N_\Omega(x_*)=\nu(c_*).
\]
Hence
\[
T_{x_*}\partial\Omega
=
T_{c_*}\partial C
\]
as Euclidean hyperplanes, and we use this identification below.

Let
\[
S_*^\Omega
=
-DN_\Omega(x_*)|_{T_{x_*}\partial\Omega}
\]
be the shape operator of the outer boundary, with the same sign
convention as for $S_*$.

\begin{theorem}[Curvature representation of the thickness Hessian]
	\label{thm:curvature-representation}
	At every critical point $c_*$,
	\begin{equation}
	\boxed{
		H_*
		=
		L_*S_*^\Omega L_*
		-
		L_*S_*.
	}
	\label{eq:H-curvature-representation}
	\end{equation}
	Consequently,
	\begin{equation}
	\boxed{
		K_*
		=
		d_*
		\left(
		L_*^{1/2}S_*^\Omega L_*^{1/2}
		-
		S_*
		\right).
	}
	\label{eq:K-curvature-representation}
	\end{equation}
\end{theorem}

\begin{proof}
	For
	\[
	g=\nabla_{\partial C}d,
	\qquad
	w=(I-dS)^{-1}g,
	\]
	the outward unit normal of the outer boundary is
	\[
	N_\Omega(\Phi(c))
	=
	\frac{\nu(c)-w(c)}
	{\sqrt{1+|w(c)|^2}}.
	\]
	At $c_*$ one has $g(c_*)=w(c_*)=0$.  Hence, for
	$v\in T_{c_*}\partial C$,
	\[
	Dw(c_*)[v]
	=
	L_*^{-1}H_*v.
	\]
	Differentiating the normal formula at $c_*$ therefore gives
	\[
	D(N_\Omega\circ\Phi)(c_*)[v]
	=
	-S_*v-L_*^{-1}H_*v.
	\]
	
	On the other hand,
	\[
	D\Phi(c_*)[v]=L_*v,
	\]
	and the Weingarten equation on $\partial\Omega$ gives
	\[
	D(N_\Omega\circ\Phi)(c_*)[v]
	=
	-S_*^\Omega L_*v.
	\]
	Equating the two expressions yields
	\[
	S_*^\Omega L_*v
	=
	S_*v+L_*^{-1}H_*v.
	\]
	Multiplication by $L_*$ gives
	\[
	L_*S_*^\Omega L_*v
	=
	L_*S_*v+H_*v,
	\]
	which proves \eqref{eq:H-curvature-representation}.
	
	Finally,
	\[
	B_*^{1/2}
	=
	d_*^{1/2}L_*^{-1/2}.
	\]
	Using the fact that $L_*$ is a function of $S_*$ and hence commutes
	with $S_*$, we obtain
	\[
	\begin{aligned}
	K_*
	&=
	d_*L_*^{-1/2}
	\left(
	L_*S_*^\Omega L_*-L_*S_*
	\right)
	L_*^{-1/2}
	\\
	&=
	d_*
	\left(
	L_*^{1/2}S_*^\Omega L_*^{1/2}
	-
	S_*
	\right),
	\end{aligned}
	\]
	which is \eqref{eq:K-curvature-representation}.
\end{proof}

The preceding identity does not require the principal directions of
the two boundaries to agree.  When they do agree, it reduces to an
especially transparent scalar formula.

\begin{corollary}[Aligned principal directions]
	\label{cor:aligned-curvature}
	Assume that $S_*$ and $S_*^\Omega$ admit a common orthonormal
	eigenbasis
	\[
	e_1,\ldots,e_m,
	\]
	with
	\[
	S_*e_i=-\kappa_i^C e_i,
	\qquad
	S_*^\Omega e_i=-\kappa_i^\Omega e_i.
	\]
	Set
	\[
	\ell_i
	=
	1+d_*\kappa_i^C.
	\]
	Then the corresponding Hessian eigenvalue is
	\begin{equation}
	\eta_i
	=
	\ell_i
	\left(
	\kappa_i^C-\ell_i\kappa_i^\Omega
	\right),
	\label{eq:aligned-H-eigenvalue}
	\end{equation}
	the preconditioned eigenvalue is
	\begin{equation}
	\boxed{
		\mu_i
		=
		d_*
		\left[
		\kappa_i^C-
		(1+d_*\kappa_i^C)\kappa_i^\Omega
		\right],
	}
	\label{eq:aligned-mu}
	\end{equation}
	and the corresponding return multiplier is
	\begin{equation}
	\boxed{
		\Lambda_i
		=
		(1+d_*\kappa_i^C)
		(1-d_*\kappa_i^\Omega).
	}
	\label{eq:aligned-multiplier}
	\end{equation}
\end{corollary}

\begin{proof}
	Since
	\[
	L_*e_i=\ell_i e_i,
	\]
	equation~\eqref{eq:H-curvature-representation} gives
	\[
	H_*e_i
	=
	\left(
	-\ell_i^2\kappa_i^\Omega
	+
	\ell_i\kappa_i^C
	\right)e_i,
	\]
	which proves \eqref{eq:aligned-H-eigenvalue}.  Since the corresponding
	eigenvalue of $B_*$ is $d_*/\ell_i$,
	\[
	\mu_i
	=
	\frac{d_*}{\ell_i}\eta_i,
	\]
	and \eqref{eq:aligned-mu} follows.  Finally,
	\[
	\Lambda_i=1+\mu_i
	=
	(1+d_*\kappa_i^C)(1-d_*\kappa_i^\Omega).
	\]
\end{proof}

The curvature formula becomes particularly geometric when compared
with the parallel surface of the convex core.  Its principal
curvature at distance $d_*$ is
\[
\kappa_i^{\mathrm{par}}
=
\frac{\kappa_i^C}
{1+d_*\kappa_i^C}.
\]
Hence
\begin{equation}
\eta_i
=
(1+d_*\kappa_i^C)^2
\left(
\kappa_i^{\mathrm{par}}
-
\kappa_i^\Omega
\right).
\label{eq:H-parallel-curvature}
\end{equation}

\begin{corollary}[Morse index as a curvature count]
	\label{cor:morse-curvature-count}
	Under the alignment hypothesis,
	\begin{equation}
	\operatorname{ind}_M(c_*)
	=
	\#
	\left\{
	i:
	\kappa_i^\Omega
	>
	\kappa_i^{\mathrm{par}}
	\right\},
	\label{eq:morse-curvature-count}
	\end{equation}
	with multiplicities counted.
\end{corollary}

Thus, when the principal frames are aligned, a negative Hessian
direction is precisely a direction in which the outer boundary is
more curved than the parallel surface of the core at the same
normal distance.

\begin{remark}[The noncommuting case]
	\label{rem:noncommuting-curvature}
	Formula~\eqref{eq:K-curvature-representation} remains exact without
	alignment.  What is lost is the reduction to pairwise scalar
	principal-curvature comparisons.  In the noncommuting case, the
	spectrum is governed by the self-adjoint operator
	\[
	L_*^{1/2}S_*^\Omega L_*^{1/2}-S_*,
	\]
	so the relative orientation of the two principal frames becomes part
	of the spectral problem.
\end{remark}

% ---------------------------------------------------------
\subsection{Stable and unstable spectral dimensions}
\label{subsec:stable-unstable-spectrum}
% ---------------------------------------------------------

Inertia alone does not determine stability.  A negative eigenvalue
$\mu$ of $K_*$ produces the multiplier
\[
1+\mu,
\]
which is stable precisely when
\[
-2<\mu<0.
\]
The strict-ascent structure supplies exactly this additional
information.

\begin{lemma}[Negative preconditioned eigenvalues in the ascent regime]
	\label{lem:negative-mu-ascent}
	Assume that the return map satisfies, in a neighbourhood of $c_*$,
	the strict ascent estimate
	\begin{equation}
	d(F(c))-d(c)
	\ge
	\gamma
	|\nabla_{\partial C}d(c)|^2
	\label{eq:local-strict-ascent}
	\end{equation}
	for some $\gamma>0$.  If $\mu<0$ is an eigenvalue of $K_*$, then
	\begin{equation}
	-2<\mu<0.
	\label{eq:negative-mu-range}
	\end{equation}
\end{lemma}

\begin{proof}
	Let $y\neq0$ satisfy
	\[
	K_*y=\mu y
	\]
	and set
	\[
	v=B_*^{1/2}y.
	\]
	By \eqref{eq:BH-K-similarity},
	\[
	B_*H_*v=\mu v.
	\]
	Since $B_*$ is invertible,
	\[
	H_*v=\mu B_*^{-1}v,
	\]
	and therefore
	\begin{equation}
	\langle H_*v,v\rangle
	=
	\mu
	\langle B_*^{-1}v,v\rangle.
	\label{eq:H-sign-mu}
	\end{equation}
	Thus $\mu<0$ implies
	\[
	\langle H_*v,v\rangle<0.
	\]
	
	Choose Riemannian normal coordinates centered at $c_*$ and consider
	\[
	c(s)=\exp_{c_*}(sv).
	\]
	Since $c_*$ is critical,
	\[
	d(c(s))
	=
	d(c_*)
	+
	\frac{s^2}{2}
	\langle H_*v,v\rangle
	+
	o(s^2).
	\]
	Moreover,
	\[
	DF(c_*)v=(1+\mu)v,
	\]
	so in the same normal coordinates
	\[
	F(c(s))
	=
	c_*
	+
	s(1+\mu)v
	+
	o(s).
	\]
	Hence
	\begin{equation}
	d(F(c(s)))-d(c(s))
	=
	\frac{s^2}{2}
	\left((1+\mu)^2-1\right)
	\langle H_*v,v\rangle
	+
	o(s^2).
	\label{eq:increment-eigendirection}
	\end{equation}
	
	Also,
	\[
	\nabla_{\partial C}d(c(s))
	=
	sH_*v+o(s),
	\]
	and therefore
	\[
	|\nabla_{\partial C}d(c(s))|^2
	=
	s^2|H_*v|^2+o(s^2).
	\]
	Substitution into \eqref{eq:local-strict-ascent}, followed by division
	by $s^2$ and passage to the limit, gives
	\[
	\frac12
	\left((1+\mu)^2-1\right)
	\langle H_*v,v\rangle
	\ge
	\gamma |H_*v|^2
	>
	0.
	\]
	Because
	\[
	\langle H_*v,v\rangle<0,
	\]
	we must have
	\[
	(1+\mu)^2-1<0.
	\]
	Hence
	\[
	|1+\mu|<1,
	\]
	which is equivalent to
	\[
	-2<\mu<0.
	\]
\end{proof}

\begin{remark}[The role of the value $-2$]
	\label{rem:minus-two-local}
	The threshold $-2$ is obtained here from the local strict-ascent
	property itself; the global spectral sufficient condition of
	Section~\ref{sec:framework} is not required.  Thus at a Morse
	critical point strict ascent automatically places every negative
	preconditioned eigenvalue on the stable side of the multiplier
	threshold $-1$.
\end{remark}

We may now state the principal spectral result.

\begin{theorem}[Morse index--spectral dimension correspondence]
	\label{thm:morse-spectral-correspondence}
	Let $c_*$ be a Morse critical point of the thickness function and
	assume that \eqref{eq:local-strict-ascent} holds in a neighbourhood
	of $c_*$.  Set
	\[
	\lambda_M
	=
	\operatorname{ind}_M(c_*).
	\]
	Then $DF(c_*)$ has exactly $\lambda_M$ multipliers in the open unit
	disk and exactly $m-\lambda_M$ multipliers outside the closed unit
	disk.  In particular,
	\begin{equation}
	\boxed{
		\dim E^s(c_*)
		=
		\lambda_M,
	}
	\label{eq:stable-dimension}
	\end{equation}
	and
	\begin{equation}
	\boxed{
		\dim E^u(c_*)
		=
		m-\lambda_M.
	}
	\label{eq:unstable-dimension}
	\end{equation}
	Moreover,
	\begin{equation}
	\sigma(DF(c_*))
	\cap
	\{z\in\mathbb C:|z|=1\}
	=
	\varnothing.
	\label{eq:spectral-hyperbolicity}
	\end{equation}
\end{theorem}

\begin{proof}
	By Proposition~\ref{prop:inertia-preservation}, $K_*$ has exactly
	$\lambda_M$ negative eigenvalues and $m-\lambda_M$ positive
	eigenvalues.
	
	If $\mu_j<0$, Lemma~\ref{lem:negative-mu-ascent} gives
	\[
	-2<\mu_j<0,
	\]
	so
	\[
	|\Lambda_j|
	=
	|1+\mu_j|
	<1.
	\]
	If $\mu_j>0$, then
	\[
	\Lambda_j=1+\mu_j>1.
	\]
	Since $c_*$ is Morse, zero is not an eigenvalue of $K_*$, and these
	two cases exhaust the spectrum.  The asserted dimensions and
	spectral hyperbolicity follow.
\end{proof}

\begin{corollary}[Spectral index formula]
	\label{cor:spectral-index-formula}
	Under the hypotheses of
	Theorem~\ref{thm:morse-spectral-correspondence},
	\begin{equation}
	\operatorname{ind}_M(c_*)
	=
	\#
	\left\{
	\Lambda\in\sigma(DF(c_*)):
	|\Lambda|<1
	\right\},
	\label{eq:spectral-index-formula}
	\end{equation}
	where algebraic multiplicities are counted.
\end{corollary}

Thus the Morse index has a direct dynamical interpretation: it is the
number of linearly stable spectral directions of the exact geometric
return map.

% ---------------------------------------------------------
\subsection{Dynamical classification of Morse critical points}
\label{subsec:morse-dynamical-classification}
% ---------------------------------------------------------

The preceding theorem immediately gives the local classification of
Morse equilibria in the strict-ascent regime.

\begin{corollary}[Morse type and return-map stability]
	\label{cor:morse-return-stability}
	Let $c_*$ satisfy the hypotheses of
	Theorem~\ref{thm:morse-spectral-correspondence}. Then:
	\begin{enumerate}
		\item if
		\[
		\operatorname{ind}_M(c_*)=0,
		\]
		then $c_*$ is a repelling hyperbolic fixed point;
		
		\item if
		\[
		\operatorname{ind}_M(c_*)=m,
		\]
		then $c_*$ is an attracting hyperbolic fixed point;
		
		\item if
		\[
		0<\operatorname{ind}_M(c_*)<m,
		\]
		then $c_*$ has both stable and unstable spectral directions and is
		a hyperbolic saddle in the endomorphism sense.
	\end{enumerate}
\end{corollary}

Thus the ascent character of the return dynamics reverses the usual
stability picture associated with gradient descent:
\[
\begin{array}{c|c|c|c}
\text{Morse type}
&
\operatorname{ind}_M(c_*)
&
\dim E^s(c_*)
&
\text{return dynamics}
\\ \hline
\text{minimum} & 0 & 0 & \text{repeller}
\\
\text{saddle}
&1,\ldots,m-1
&\operatorname{ind}_M(c_*)
&\text{saddle}
\\
\text{maximum}
&m&m&\text{attractor}.
\end{array}
\]

\begin{remark}[What requires the ascent hypothesis]
	\label{rem:what-needs-ascent}
	The inertia identity
	\[
	\operatorname{inertia}(K_*)
	=
	\operatorname{inertia}(H_*)
	\]
	and the curvature identities
	\eqref{eq:H-curvature-representation}--%
	\eqref{eq:K-curvature-representation}
	hold independently of the ascent regime.  The stability
	classification does not.  Without strict ascent, a negative
	eigenvalue of $K_*$ may satisfy
	\[
	\mu\le-2,
	\]
	in which case the multiplier $1+\mu$ has modulus at least one.
	Thus the Morse index continues to describe the sign structure of the
	thickness Hessian, but no longer determines the stable dimension.
\end{remark}

\begin{remark}[Spectral spaces versus invariant manifolds]
	\label{rem:spectral-versus-manifolds}
	Theorem~\ref{thm:morse-spectral-correspondence} concerns spectral
	subspaces of the derivative.  The existence and dimensions of
	corresponding local invariant manifolds are treated separately in
	Section~\ref{sec:invariant-manifolds}.  In particular, unstable
	manifold statements require additional care when the return map is
	not locally invertible.
\end{remark}

\begin{remark}[Geometry--dynamics--topology dictionary]
	\label{rem:geometry-dynamics-topology}
	At a Morse equilibrium,
	\[
	\boxed{
		\operatorname{ind}_M(c_*)
		=
		n_-(H_*)
		=
		n_-(K_*)
		=
		\dim E^s(c_*).
	}
	\]
	The first quantity belongs to the Morse structure of the thickness
	landscape, the second to its local differential geometry, the third
	incorporates the curvature of the geometric return construction, and
	the fourth is dynamical.  Under principal-direction alignment,
	Corollary~\ref{cor:morse-curvature-count} adds the geometric
	interpretation
	\[
	\operatorname{ind}_M(c_*)
	=
	\#
	\left\{
	i:
	\kappa_i^\Omega
	>
	\kappa_i^{\mathrm{par}}
	\right\}.
	\]
	This is the basic geometry--dynamics--topology correspondence used
	throughout the remainder of the paper.
\end{remark}
% =========================================================
% =========================================================
\section{Morse Inequalities and Equilibrium Counts}
\label{sec:morse-inequalities}
% =========================================================

We now translate the Morse-theoretic constraints on the thickness
landscape into constraints on the equilibria of the return map.
Because the convex-core boundary is diffeomorphic to a sphere, the
topological side of the Morse inequalities is particularly explicit.

Throughout this section,
\[
m=\dim\partial C=N-1,
\]
and the thickness function $d$ is assumed to be Morse.

% ---------------------------------------------------------
\subsection{Morse polynomial and Poincar\'e polynomial}
\label{subsec:morse-poincare-polynomials}
% ---------------------------------------------------------

For each
\[
k\in\{0,\ldots,m\},
\]
let
\begin{equation}
m_k(d)
:=
\#
\left\{
c_*\in\operatorname{Crit}(d):
\operatorname{ind}_M(c_*)=k
\right\}.
\label{eq:morse-count}
\end{equation}
Since $d$ is Morse and $\partial C$ is compact, each $m_k(d)$ is
finite.

\begin{definition}[Morse polynomial]
	\label{def:morse-polynomial}
	The Morse polynomial of the thickness landscape is
	\begin{equation}
	M_d(t)
	=
	\sum_{k=0}^{m}m_k(d)t^k.
	\label{eq:morse-polynomial}
	\end{equation}
\end{definition}

Let
\[
b_k(\partial C)
=
\dim_{\mathbb F}H_k(\partial C;\mathbb F)
\]
denote the Betti numbers with coefficients in a fixed field
$\mathbb F$.

\begin{definition}[Poincar\'e polynomial]
	\label{def:poincare-polynomial}
	The Poincar\'e polynomial of the phase space is
	\begin{equation}
	P_{\partial C}(t)
	=
	\sum_{k=0}^{m}b_k(\partial C)t^k.
	\label{eq:poincare-polynomial}
	\end{equation}
\end{definition}

Classical Morse theory gives the polynomial form of the Morse
inequalities; see, for example,
\cite{milnor1963morse,bott1982morse,nicolaescu2011morse}.

\begin{theorem}[Morse polynomial relation]
	\label{thm:morse-polynomial-relation}
	There exists a polynomial
	\begin{equation}
	Q(t)
	=
	\sum_{k=0}^{m-1}q_k t^k,
	\qquad
	q_k\in\mathbb Z_{\ge0},
	\label{eq:morse-correction-polynomial}
	\end{equation}
	such that
	\begin{equation}
	M_d(t)
	=
	P_{\partial C}(t)
	+
	(1+t)Q(t).
	\label{eq:morse-polynomial-relation}
	\end{equation}
\end{theorem}

The coefficients of $Q$ measure the excess of the critical-point
structure over the homological minimum forced by the topology of the
phase space.

\begin{corollary}[Weak Morse inequalities]
	\label{cor:weak-morse-inequalities}
	For every $k=0,\ldots,m$,
	\begin{equation}
	m_k(d)
	\ge
	b_k(\partial C).
	\label{eq:weak-morse}
	\end{equation}
	In particular,
	\begin{equation}
	\#\operatorname{Crit}(d)
	=
	\sum_{k=0}^{m}m_k(d)
	\ge
	\sum_{k=0}^{m}b_k(\partial C).
	\label{eq:critical-lower-bound}
	\end{equation}
\end{corollary}

The strong Morse inequalities can equivalently be written as
\begin{equation}
\sum_{j=0}^{k}
(-1)^{k-j}m_j(d)
\ge
\sum_{j=0}^{k}
(-1)^{k-j}b_j(\partial C),
\qquad
0\le k<m,
\label{eq:strong-morse}
\end{equation}
with equality for $k=m$.  The terminal equality will be interpreted
dynamically in Section~\ref{sec:euler-fixed-index}.

% ---------------------------------------------------------
\subsection{Spherical phase space and coefficient constraints}
\label{subsec:spherical-morse-counts}
% ---------------------------------------------------------

By Proposition~\ref{prop:core-boundary-sphere},
\[
\partial C\cong\mathbb S^m.
\]
Hence
\[
b_0=b_m=1,
\qquad
b_k=0
\quad
(0<k<m),
\]
and therefore
\begin{equation}
P_{\partial C}(t)
=
1+t^m.
\label{eq:sphere-poincare-section5}
\end{equation}
The Morse polynomial relation becomes
\begin{equation}
\boxed{
	M_d(t)
	=
	1+t^m+(1+t)Q(t),
}
\label{eq:spherical-morse-polynomial}
\end{equation}
where
\[
Q(t)
=
\sum_{k=0}^{m-1}q_k t^k,
\qquad
q_k\ge0.
\]

The coefficients of $Q$ yield more precise information than the
weak inequalities alone.

\begin{proposition}[Coefficient form of the spherical Morse relations]
	\label{prop:morse-coefficient-relations}
	Set
	\[
	q_{-1}=q_m=0.
	\]
	Then, for every $k=0,\ldots,m$,
	\begin{equation}
	m_k(d)
	=
	b_k(\mathbb S^m)+q_{k-1}+q_k.
	\label{eq:general-coefficient-formula}
	\end{equation}
	Equivalently,
	\begin{align}
	m_0(d)
	&=
	1+q_0,
	\label{eq:m0-q}
	\\
	m_k(d)
	&=
	q_{k-1}+q_k,
	\qquad
	1\le k\le m-1,
	\label{eq:mk-q}
	\\
	m_m(d)
	&=
	1+q_{m-1}.
	\label{eq:mm-q}
	\end{align}
\end{proposition}

\begin{proof}
	Expand
	\[
	(1+t)Q(t)
	=
	\sum_{k=0}^{m-1}q_kt^k
	+
	\sum_{k=1}^{m}q_{k-1}t^k
	\]
	and compare coefficients in
	\eqref{eq:spherical-morse-polynomial}.
\end{proof}

Thus the correction coefficient $q_k$ contributes algebraically to
the critical-point counts in the two adjacent Morse indices
$k$ and $k+1$.

An immediate global consequence is the following.

\begin{corollary}[Parity and total number of equilibria]
	\label{cor:even-critical-count}
	Every Morse thickness landscape on $\partial C$ satisfies
	\begin{equation}
	\boxed{
		\#\operatorname{Crit}(d)
		=
		2+2Q(1)
		=
		2+2\sum_{k=0}^{m-1}q_k.
	}
	\label{eq:total-critical-count}
	\end{equation}
	Consequently,
	\begin{equation}
	\#\operatorname{Crit}(d)
	\equiv0\pmod 2.
	\label{eq:critical-count-even}
	\end{equation}
	Since
	\[
	\operatorname{Fix}(F)
	=
	\operatorname{Crit}(d),
	\]
	the return map has an even number of equilibria whenever the
	thickness function is Morse.
\end{corollary}

\begin{proof}
	Evaluate \eqref{eq:spherical-morse-polynomial} at $t=1$:
	\[
	M_d(1)
	=
	2+2Q(1).
	\]
	Since
	\[
	M_d(1)
	=
	\sum_{k=0}^{m}m_k(d)
	=
	\#\operatorname{Crit}(d),
	\]
	the result follows.
\end{proof}

\begin{corollary}[Minimal and nonminimal landscapes]
	\label{cor:minimal-nonminimal-landscapes}
	Every Morse thickness landscape has at least two critical points.
	Moreover:
	\begin{enumerate}
		\item
		$\#\operatorname{Crit}(d)=2$ if and only if $Q\equiv0$;
		
		\item
		if the landscape is not perfect, then
		\[
		\#\operatorname{Crit}(d)\ge4.
		\]
	\end{enumerate}
\end{corollary}

This strengthens the elementary existence of a minimum and a
maximum: every departure from the homologically minimal configuration
introduces at least two additional critical points.

\begin{definition}[Perfect Morse thickness landscape]
	\label{def:perfect-morse-landscape}
	A Morse thickness function is called \emph{perfect} if
	\[
	M_d(t)
	=
	P_{\partial C}(t).
	\]
	For the convex-core phase space this is equivalent to
	\begin{equation}
	M_d(t)
	=
	1+t^m.
	\label{eq:perfect-sphere-polynomial}
	\end{equation}
\end{definition}

\begin{corollary}[Perfect landscape on the convex-core sphere]
	\label{cor:perfect-landscape}
	A perfect Morse thickness landscape has exactly two critical points:
	one minimum and one maximum. Thus
	\[
	m_0=m_m=1,
	\qquad
	m_k=0
	\quad
	(0<k<m).
	\]
\end{corollary}

\begin{remark}[Necessity versus realizability]
	\label{rem:morse-count-realizability}
	Equations~\eqref{eq:m0-q}--\eqref{eq:mm-q} are necessary
	consequences of the Morse inequalities for every admissible Morse
	thickness landscape. They should not be interpreted as asserting that
	every polynomial $Q$ with nonnegative integer coefficients is
	realized by a thickness function arising from a domain in
	$\mathcal O_C$. Such realizability is constrained by the geometry of
	the admissible class and is related to the inverse problem discussed
	later in the paper.
\end{remark}

% ---------------------------------------------------------
\subsection{Dynamical equilibrium polynomial}
\label{subsec:equilibrium-polynomial}
% ---------------------------------------------------------

We now impose the strict-ascent hypotheses of
Theorem~\ref{thm:morse-spectral-correspondence}. For
$k=0,\ldots,m$, define
\begin{equation}
e_k(F)
:=
\#
\left\{
c_*\in\operatorname{Fix}(F):
\dim E^s(c_*)=k
\right\}.
\label{eq:equilibrium-count}
\end{equation}

\begin{definition}[Stable-index equilibrium polynomial]
	\label{def:equilibrium-polynomial}
	The stable-index equilibrium polynomial of the return map is
	\begin{equation}
	E_F(t)
	=
	\sum_{k=0}^{m}e_k(F)t^k.
	\label{eq:equilibrium-polynomial}
	\end{equation}
\end{definition}

The Morse-index/spectral-dimension correspondence gives an exact
identity.

\begin{theorem}[Morse--dynamical polynomial identity]
	\label{thm:morse-dynamical-polynomial}
	Assume that $d$ is Morse and that the return map lies in the strict
	ascent regime. Then
	\begin{equation}
	e_k(F)=m_k(d),
	\qquad
	k=0,\ldots,m.
	\label{eq:e-m-equality}
	\end{equation}
	Consequently,
	\begin{equation}
	\boxed{
		E_F(t)=M_d(t)
	}
	\label{eq:E-equals-M}
	\end{equation}
	and hence
	\begin{equation}
	\boxed{
		E_F(t)
		=
		1+t^m+(1+t)Q(t).
	}
	\label{eq:dynamical-morse-polynomial}
	\end{equation}
\end{theorem}

\begin{proof}
	By
	\[
	\operatorname{Fix}(F)
	=
	\operatorname{Crit}(d)
	\]
	and
	Theorem~\ref{thm:morse-spectral-correspondence}, every fixed point
	of Morse index $k$ has stable spectral dimension exactly $k$.
	Thus the sets counted by $e_k(F)$ and $m_k(d)$ coincide.
	Equation~\eqref{eq:dynamical-morse-polynomial} follows from
	\eqref{eq:spherical-morse-polynomial}.
\end{proof}

This is the dynamical form of the Morse inequalities: topology
constrains not only the total number of equilibria, but their
distribution among stable spectral dimensions.

Combining Theorem~\ref{thm:morse-dynamical-polynomial} with
Proposition~\ref{prop:morse-coefficient-relations} gives
\begin{align}
e_0(F)
&=
1+q_0,
\label{eq:e0-q}
\\
e_k(F)
&=
q_{k-1}+q_k,
\qquad
1\le k\le m-1,
\label{eq:ek-q}
\\
e_m(F)
&=
1+q_{m-1}.
\label{eq:em-q}
\end{align}

\begin{corollary}[Existence of a repeller and an attractor]
	\label{cor:repeller-attractor-existence}
	Under the strict-ascent hypotheses, the return map has at least one
	repelling Morse minimum and at least one attracting Morse maximum:
	\[
	e_0(F)\ge1,
	\qquad
	e_m(F)\ge1.
	\]
\end{corollary}

\begin{corollary}[Perfect dynamics]
	\label{cor:perfect-dynamics}
	If the thickness landscape is perfect, then the strict-ascent return
	dynamics has exactly two equilibria:
	one repeller and one attractor.
	There are no saddle equilibria.
\end{corollary}

\begin{remark}[Adjacent stable dimensions]
	\label{rem:adjacent-stable-dimensions}
	At the level of the Morse polynomial, each coefficient $q_k$
	contributes to the equilibrium counts in the two adjacent stable
	dimensions $k$ and $k+1$. This is an algebraic consequence of the
	factor $(1+t)$ in the Morse relation; it does not assert a canonical
	pairing of individual equilibria.
\end{remark}

% ---------------------------------------------------------
\subsection{Low-dimensional consequences}
\label{subsec:low-dimensional-counts}
% ---------------------------------------------------------

The general coefficient relations become particularly transparent
in low dimensions.

\paragraph{The circle: $m=1$.}
Here
\[
Q(t)=q_0,
\]
and therefore
\[
m_0=m_1=1+q_0.
\]
Under strict ascent,
\begin{equation}
\boxed{
	\#\{\text{repellers}\}
	=
	\#\{\text{attractors}\}.
}
\label{eq:circle-count-balance}
\end{equation}

\paragraph{The two-sphere: $m=2$.}
Writing
\[
Q(t)=q_0+q_1t,
\]
one obtains
\[
m_0=1+q_0,
\qquad
m_1=q_0+q_1,
\qquad
m_2=1+q_1.
\]
Eliminating $q_0$ and $q_1$ gives
\begin{equation}
m_1
=
m_0+m_2-2.
\label{eq:S2-morse-balance}
\end{equation}
Hence, in the strict-ascent regime,
\begin{equation}
\boxed{
	\#\{\text{saddles}\}
	=
	\#\{\text{repellers}\}
	+
	\#\{\text{attractors}\}
	-2.
}
\label{eq:S2-dynamical-balance}
\end{equation}

\paragraph{The three-sphere: $m=3$.}
Writing
\[
Q(t)=q_0+q_1t+q_2t^2,
\]
gives
\[
m_0=1+q_0,
\qquad
m_1=q_0+q_1,
\qquad
m_2=q_1+q_2,
\qquad
m_3=1+q_2.
\]
In particular,
\begin{equation}
m_0-m_1+m_2-m_3=0,
\label{eq:S3-euler-balance}
\end{equation}
or equivalently,
\begin{equation}
m_0+m_2
=
m_1+m_3.
\label{eq:S3-count-balance}
\end{equation}
Under strict ascent, the same formulas hold for the numbers of
equilibria of stable dimensions $0,1,2,$ and $3$.

\begin{remark}[Geometry versus topology]
	\label{rem:geometry-topology-counts}
	The identities in this section do not depend on the numerical values
	of the return multipliers.  The local geometry of the pair
	$(C,\Omega)$ determines spectral rates through $B_*$, $K_*$ and,
	by Section~\ref{sec:morse-spectrum}, the curvature operators of the
	two boundaries.  Morse topology instead constrains the global
	distribution of equilibria among the possible stable dimensions.
	Thus the return dynamics separates naturally into two complementary
	levels:
	\[
	\boxed{
		\begin{array}{c}
		\text{geometry determines local spectral data},
		\\[1mm]
		\text{Morse topology constrains global spectral counts}.
		\end{array}}
	\]
\end{remark}
% =========================================================
% =========================================================
\section{Euler Characteristic and Fixed-Point Index}
\label{sec:euler-fixed-index}
% =========================================================

The Morse polynomial relation obtained in
Section~\ref{sec:morse-inequalities} yields an alternating balance
among the critical points of the thickness function.  We now show
that the same balance appears as the fixed-point-index balance of the
exact return map.

This section requires only that the thickness function be Morse,
except where the strict-ascent regime is invoked explicitly.  In
particular, the local fixed-point-index formula and the degree-one
conclusion do not require the stability classification of
Section~\ref{sec:morse-spectrum}.

% ---------------------------------------------------------
\subsection{Morse--Euler balance}
\label{subsec:morse-euler}
% ---------------------------------------------------------

Evaluating the Morse polynomial relation
\[
M_d(t)
=
P_{\partial C}(t)
+
(1+t)Q(t)
\]
at $t=-1$ eliminates the correction term and gives the classical
Euler relation.

\begin{theorem}[Morse--Euler identity]
	\label{thm:morse-euler}
	Let $d$ be a Morse thickness function on $\partial C$. Then
	\begin{equation}
	\boxed{
		\sum_{c_*\in\operatorname{Crit}(d)}
		(-1)^{\operatorname{ind}_M(c_*)}
		=
		\chi(\partial C).
	}
	\label{eq:morse-euler}
	\end{equation}
	Equivalently,
	\begin{equation}
	\sum_{k=0}^{m}
	(-1)^k m_k(d)
	=
	\chi(\partial C).
	\label{eq:morse-euler-counts}
	\end{equation}
\end{theorem}

\begin{proof}
	By Theorem~\ref{thm:morse-polynomial-relation},
	\[
	M_d(t)
	=
	P_{\partial C}(t)+(1+t)Q(t).
	\]
	Setting $t=-1$ gives
	\[
	M_d(-1)=P_{\partial C}(-1).
	\]
	The left-hand side is
	\[
	\sum_{k=0}^{m}(-1)^k m_k(d),
	\]
	while the right-hand side is $\chi(\partial C)$.
\end{proof}

Since
\[
\partial C\cong\mathbb S^m,
\]
one has
\begin{equation}
\chi(\partial C)
=
1+(-1)^m.
\label{eq:euler-sphere-section6}
\end{equation}

% ---------------------------------------------------------
\subsection{Local fixed-point index of the return map}
\label{subsec:local-fixed-index}
% ---------------------------------------------------------

Recall that an isolated fixed point $c_*$ of a $C^1$ map is
nondegenerate as a fixed point if
\[
1\notin\sigma(DF(c_*)),
\]
or equivalently,
\[
\det(I-DF(c_*))\neq0.
\]
For an isolated nondegenerate fixed point of a $C^1$ map, the local
fixed-point index is
\[
\operatorname{ind}_{\mathrm{fix}}(F,c_*)
=
\operatorname{sgn}\det(I-DF(c_*)).
\]
This is the standard local formula entering the Lefschetz--Hopf
fixed-point theorem; see, for example,
\cite{brown1971lefschetz}.

\begin{proposition}[Morse points are nondegenerate fixed points]
	\label{prop:morse-fixed-nondegenerate}
	Let $c_*$ be a Morse critical point of $d$. Then $c_*$ is a
	nondegenerate fixed point of $F$:
	\begin{equation}
	\det(I-DF(c_*))\neq0.
	\label{eq:fixed-nondegeneracy}
	\end{equation}
\end{proposition}

\begin{proof}
	The exact linearization gives
	\[
	DF(c_*)=I+B_*H_*.
	\]
	Hence
	\begin{equation}
	I-DF(c_*)
	=
	-B_*H_*.
	\label{eq:I-minus-DF}
	\end{equation}
	Since $B_*$ is positive definite and $H_*$ is nonsingular at a
	Morse critical point, $-B_*H_*$ is invertible.
\end{proof}

\begin{remark}[Two different notions of nondegeneracy]
	\label{rem:two-nondegeneracies}
	Fixed-point nondegeneracy means
	\[
	1\notin\sigma(DF(c_*)).
	\]
	This should not be confused with local invertibility of the return
	map, which requires
	\[
	0\notin\sigma(DF(c_*)).
	\]
	In terms of the eigenvalues $\mu_j$ of $K_*$, the first condition is
	equivalent to
	\[
	\mu_j\neq0,
	\]
	and is therefore guaranteed by the Morse hypothesis, whereas local
	invertibility requires
	\[
	\mu_j\neq-1.
	\]
	Thus a Morse equilibrium may be a nondegenerate fixed point even when
	$DF(c_*)$ has a zero multiplier.
\end{remark}

We can now compute the local index exactly.

\begin{theorem}[Fixed-point index from the Morse index]
	\label{thm:fixed-index-morse-index}
	Let $c_*$ be a Morse critical point of $d$. Then
	\begin{equation}
	\boxed{
		\operatorname{ind}_{\mathrm{fix}}(F,c_*)
		=
		(-1)^{m+\operatorname{ind}_M(c_*)}.
	}
	\label{eq:fixed-index-morse-index}
	\end{equation}
\end{theorem}

\begin{proof}
	From \eqref{eq:I-minus-DF},
	\[
	\det(I-DF(c_*))
	=
	\det(-B_*H_*).
	\]
	Since the tangent space has dimension $m$,
	\[
	\det(-B_*H_*)
	=
	(-1)^m
	\det(B_*)
	\det(H_*).
	\]
	Because $B_*$ is positive definite,
	\[
	\det(B_*)>0.
	\]
	If
	\[
	\lambda_M
	=
	\operatorname{ind}_M(c_*),
	\]
	then $H_*$ has exactly $\lambda_M$ negative eigenvalues and no zero
	eigenvalues. Hence
	\[
	\operatorname{sgn}\det(H_*)
	=
	(-1)^{\lambda_M}.
	\]
	Therefore
	\[
	\operatorname{sgn}\det(I-DF(c_*))
	=
	(-1)^{m+\lambda_M},
	\]
	which proves the result.
\end{proof}

\begin{remark}[Curvature magnitudes and the fixed-point index]
	\label{rem:index-curvature-magnitude}
	The local fixed-point index depends on the curvature preconditioner
	$B_*$ only through the positivity of its determinant.  Curvature may
	substantially modify the multipliers and eigendirections of the
	return map, but it cannot alter the sign of the local index as long
	as the Morse type is unchanged.
\end{remark}

Under the strict-ascent hypotheses, the formula acquires a purely
dynamical interpretation.

\begin{corollary}[Fixed-point index from spectral dimensions]
	\label{cor:index-spectral-dimensions}
	Assume in addition the hypotheses of
	Theorem~\ref{thm:morse-spectral-correspondence}. Then
	\begin{equation}
	\operatorname{ind}_{\mathrm{fix}}(F,c_*)
	=
	(-1)^{m+\dim E^s(c_*)}.
	\label{eq:index-stable-dimension}
	\end{equation}
	Equivalently,
	\begin{equation}
	\boxed{
		\operatorname{ind}_{\mathrm{fix}}(F,c_*)
		=
		(-1)^{\dim E^u(c_*)}.
	}
	\label{eq:index-unstable-dimension}
	\end{equation}
\end{corollary}

\begin{proof}
	By Theorem~\ref{thm:morse-spectral-correspondence},
	\[
	\dim E^s(c_*)
	=
	\operatorname{ind}_M(c_*),
	\]
	and
	\[
	\dim E^u(c_*)
	=
	m-\operatorname{ind}_M(c_*).
	\]
	Since
	\[
	m+\operatorname{ind}_M(c_*)
	\equiv
	m-\operatorname{ind}_M(c_*)
	\pmod 2,
	\]
	the two formulas follow from
	Theorem~\ref{thm:fixed-index-morse-index}.
\end{proof}

Thus, in the strict-ascent regime,
\[
\begin{array}{c|c|c}
\text{Morse type}
&
\dim E^u(c_*)
&
\operatorname{ind}_{\mathrm{fix}}(F,c_*)
\\ \hline
\text{minimum / repeller}
&
m
&
(-1)^m
\\
\text{saddle}
&
m-\operatorname{ind}_M(c_*)
&
(-1)^{m+\operatorname{ind}_M(c_*)}
\\
\text{maximum / attractor}
&
0
&
+1.
\end{array}
\]

The curvature interpretation of Section~\ref{sec:morse-spectrum}
gives an additional geometric formulation.

\begin{corollary}[Fixed-point index from aligned curvature data]
	\label{cor:index-curvature-count}
	Under the principal-direction alignment hypothesis of
	Corollary~\ref{cor:morse-curvature-count},
	\begin{equation}
	\operatorname{ind}_{\mathrm{fix}}(F,c_*)
	=
	(-1)^{
		m+
		\#\{\,i:\kappa_i^\Omega>\kappa_i^{\mathrm{par}}\,\}
	}.
	\label{eq:index-curvature-count}
	\end{equation}
\end{corollary}

This formula is independent of the strict-ascent hypothesis; it uses
only the Morse character of the critical point and the aligned
curvature representation.

% ---------------------------------------------------------
\subsection{Global fixed-point-index balance}
\label{subsec:global-index-balance}
% ---------------------------------------------------------

Summing the local indices gives a global topological constraint on
the return dynamics.

\begin{theorem}[Global fixed-point-index balance]
	\label{thm:global-fixed-index-balance}
	Let $d$ be Morse. Then
	\begin{equation}
	\boxed{
		\sum_{c_*\in\operatorname{Fix}(F)}
		\operatorname{ind}_{\mathrm{fix}}(F,c_*)
		=
		\chi(\partial C).
	}
	\label{eq:global-fixed-index-balance}
	\end{equation}
\end{theorem}

\begin{proof}
	Since
	\[
	\operatorname{Fix}(F)
	=
	\operatorname{Crit}(d),
	\]
	Theorem~\ref{thm:fixed-index-morse-index} gives
	\[
	\begin{aligned}
	\sum_{c_*\in\operatorname{Fix}(F)}
	\operatorname{ind}_{\mathrm{fix}}(F,c_*)
	&=
	\sum_{c_*\in\operatorname{Crit}(d)}
	(-1)^{m+\operatorname{ind}_M(c_*)}
	\\
	&=
	(-1)^m
	\sum_{c_*\in\operatorname{Crit}(d)}
	(-1)^{\operatorname{ind}_M(c_*)}.
	\end{aligned}
	\]
	By the Morse--Euler identity,
	\[
	\sum_{c_*\in\operatorname{Crit}(d)}
	(-1)^{\operatorname{ind}_M(c_*)}
	=
	\chi(\partial C).
	\]
	Thus
	\begin{equation}
	\sum_{c_*\in\operatorname{Fix}(F)}
	\operatorname{ind}_{\mathrm{fix}}(F,c_*)
	=
	(-1)^m\chi(\partial C).
	\label{eq:general-index-balance}
	\end{equation}
	
	Since
	\[
	\partial C\cong\mathbb S^m,
	\]
	if $m$ is even then $(-1)^m=1$, whereas if $m$ is odd then
	$\chi(\mathbb S^m)=0$.  Hence
	\[
	(-1)^m\chi(\partial C)
	=
	\chi(\partial C),
	\]
	which proves \eqref{eq:global-fixed-index-balance}.
\end{proof}

\begin{remark}[The algebraic identity before spherical specialization]
	\label{rem:general-index-balance}
	The relation obtained before using
	$\partial C\cong\mathbb S^m$ is
	\[
	\sum_{c_*}
	\operatorname{ind}_{\mathrm{fix}}(F,c_*)
	=
	(-1)^m\chi(\partial C).
	\]
	For the boundary of a compact convex body this reduces automatically
	to \eqref{eq:global-fixed-index-balance}.
\end{remark}

Under strict ascent, using the equilibrium counts of
Section~\ref{sec:morse-inequalities}, the same balance becomes
\begin{equation}
\sum_{k=0}^{m}
(-1)^{m-k}e_k(F)
=
\chi(\partial C).
\label{eq:equilibrium-index-balance}
\end{equation}
Since
\[
(-1)^{m-k}=(-1)^{m+k},
\]
this is equivalent to the expression obtained directly from the
stable dimensions.

% ---------------------------------------------------------
\subsection{Lefschetz number and degree}
\label{subsec:lefschetz-degree}
% ---------------------------------------------------------

By the Lefschetz--Hopf theorem
\cite{brown1971lefschetz},
the sum of the local fixed-point indices of the isolated fixed
points equals the Lefschetz number $L(F)$.
\begin{equation}
L(F)
=
\sum_{c_*\in\operatorname{Fix}(F)}
\operatorname{ind}_{\mathrm{fix}}(F,c_*).
\label{eq:lefschetz-index-sum}
\end{equation}
Hence
\begin{equation}
L(F)
=
\chi(\partial C).
\label{eq:lefschetz-euler}
\end{equation}

For a continuous self-map of
\[
\partial C\cong\mathbb S^m,
\]
with $m\ge1$,
\begin{equation}
L(F)
=
1+(-1)^m\deg(F).
\label{eq:sphere-lefschetz-degree}
\end{equation}
Combining this with
\[
\chi(\mathbb S^m)
=
1+(-1)^m
\]
gives the following global consequence.

\begin{corollary}[Degree of the return map]
	\label{cor:return-map-degree}
	Under the standing Morse hypotheses,
	\begin{equation}
	\boxed{
		\deg(F)=1.
	}
	\label{eq:return-map-degree}
	\end{equation}
\end{corollary}

\begin{proof}
	From
	\eqref{eq:lefschetz-euler} and
	\eqref{eq:sphere-lefschetz-degree},
	\[
	1+(-1)^m\deg(F)
	=
	1+(-1)^m.
	\]
	Therefore
	\[
	\deg(F)=1.
	\]
\end{proof}

\begin{remark}[No ascent hypothesis is needed]
	\label{rem:degree-no-ascent}
	The degree-one conclusion is independent of the strict-ascent
	regime. It follows from the Morse character of the thickness
	landscape, the exact fixed-point relation
	\[
	\operatorname{Fix}(F)=\operatorname{Crit}(d),
	\]
	and the exact linearization at those fixed points.  Thus the degree
	constraint persists even when the local stability classification of
	Section~\ref{sec:morse-spectrum} is unavailable.
\end{remark}

\begin{remark}[Topological meaning]
	\label{rem:degree-topological-meaning}
	The degree-one conclusion is independent of the number of critical
	points and of the numerical values of the multipliers.  As the
	thickness landscape becomes more complicated, fixed points may
	contribute different local signs, but their total contribution is
	constrained so that the return map has the same degree as the
	identity map of the spherical phase space.
\end{remark}

% ---------------------------------------------------------
\subsection{Low-dimensional interpretation}
\label{subsec:low-dimensional-index}
% ---------------------------------------------------------

The low-dimensional balances obtained in
Section~\ref{sec:morse-inequalities} can now be interpreted as
fixed-point-index identities.

For $m=1$, repellers have local index $-1$ and attractors have local
index $+1$, so
\[
-\#\{\text{repellers}\}
+
\#\{\text{attractors}\}
=
0.
\]

For $m=2$, repellers and attractors have local index $+1$, whereas
index-one saddles have local index $-1$, so
\[
\#\{\text{repellers}\}
-
\#\{\text{saddles}\}
+
\#\{\text{attractors}\}
=
2.
\]

These are not new counting relations beyond Section~\ref{sec:morse-inequalities};
their role here is to show how the Morse balances become local
fixed-point-index balances for the discrete return map.

\begin{remark}[Two complementary index theories]
	\label{rem:two-index-theories}
	The Morse index and the local fixed-point index encode different
	amounts of information.  The Morse index
	\[
	\operatorname{ind}_M(c_*)
	\]
	counts the negative directions of the thickness Hessian and, in the
	strict-ascent regime, the stable directions of the return map.
	The fixed-point index
	\[
	\operatorname{ind}_{\mathrm{fix}}(F,c_*)
	\in\{+1,-1\}
	\]
	retains only the parity information encoded in
	\[
	\det(I-DF(c_*)).
	\]
	The relation
	\[
	\operatorname{ind}_{\mathrm{fix}}(F,c_*)
	=
	(-1)^{m+\operatorname{ind}_M(c_*)}
	\]
	describes precisely how the finer Morse information collapses to the
	local topological index of the discrete map.
\end{remark}
% =========================================================
% =========================================================
\section{Topology of the Thickness Filtration}
\label{sec:thickness-filtration}
% =========================================================

Morse theory describes a function not only through its critical
points but also through the topology of its level-set filtration.
For the thickness function this filtration acquires a direct
dynamical interpretation: its handle indices coincide with the
stable and unstable spectral dimensions of the return map.

Throughout this section, $d$ is assumed to be Morse and
\[
m=\dim\partial C=N-1.
\]
For $a\in\mathbb R$, define
\begin{equation}
L_a
:=
\{c\in\partial C:d(c)\le a\},
\label{eq:sublevel-set}
\end{equation}
and
\begin{equation}
U^a
:=
\{c\in\partial C:d(c)\ge a\}.
\label{eq:superlevel-set}
\end{equation}
Since $\partial C$ is compact,
\[
d_{\min}\le d(c)\le d_{\max},
\]
and these two filtrations interpolate between the empty set and the
whole phase space.

% ---------------------------------------------------------
\subsection{Sublevel and superlevel handle attachments}
\label{subsec:handle-filtrations}
% ---------------------------------------------------------

We first recall the standard Morse-theoretic description of the
level-set filtration; see
\cite{milnor1963morse,bott1982morse,nicolaescu2011morse}.

\begin{theorem}[Regular-level invariance]
	\label{thm:regular-level-invariance}
	Let $a<b$ be regular values of $d$. If
	\[
	d^{-1}([a,b])
	\]
	contains no critical point, then $L_a$ and $L_b$ are $C^1$-diffeomorphic.
	Equivalently, the topology of the sublevel filtration changes only
	when a critical value is crossed.
	
	The same statement applies to the superlevel filtration.
\end{theorem}

Let $c_*$ be a Morse critical point with
\[
r_*=d(c_*),
\qquad
\lambda=\operatorname{ind}_M(c_*).
\]

\begin{theorem}[Dual handle attachments]
	\label{thm:dual-handle-attachment}
	Let
	\[
	a<r_*<b
	\]
	be regular values and suppose that $c_*$ is the only critical point
	in
	\[
	d^{-1}([a,b]).
	\]
	Then:
	
	\begin{enumerate}
		\item
		the sublevel set $L_b$ is obtained from $L_a$, up to the standard
		Morse handle attachment, by attaching a handle of index
		\[
		\lambda;
		\]
		
		\item
		as the superlevel threshold is lowered from $b$ to $a$, the set
		$U^a$ is obtained from $U^b$ by attaching a handle of index
		\[
		m-\lambda.
		\]
	\end{enumerate}
\end{theorem}

\begin{proof}
	The first assertion is the classical Morse handle-attachment theorem.
	For the second, observe that
	\[
	U^a
	=
	\{d\ge a\}
	=
	\{-d\le -a\}.
	\]
	The Morse index of $-d$ at $c_*$ is
	\[
	\operatorname{ind}_M(-d,c_*)
	=
	m-\lambda,
	\]
	so the same theorem applied to $-d$ gives the superlevel statement.
\end{proof}

The corresponding relative homology groups are
\begin{equation}
H_j(L_b,L_a;\mathbb F)
\cong
\begin{cases}
\mathbb F, & j=\lambda,\\
0, & j\neq\lambda,
\end{cases}
\label{eq:sublevel-relative-homology}
\end{equation}
and
\begin{equation}
H_j(U^a,U^b;\mathbb F)
\cong
\begin{cases}
\mathbb F, & j=m-\lambda,\\
0, & j\neq m-\lambda.
\end{cases}
\label{eq:superlevel-relative-homology}
\end{equation}

Consequently,
\begin{equation}
\chi(L_b)-\chi(L_a)
=
(-1)^\lambda,
\label{eq:sublevel-euler-jump}
\end{equation}
while
\begin{equation}
\chi(U^a)-\chi(U^b)
=
(-1)^{m-\lambda}.
\label{eq:superlevel-euler-jump}
\end{equation}

% ---------------------------------------------------------
\subsection{Handle indices and spectral dimensions}
\label{subsec:handle-spectral-dimensions}
% ---------------------------------------------------------

The handle indices become dynamical once the strict-ascent regime is
imposed.

\begin{theorem}[Filtration--spectrum correspondence]
	\label{thm:filtration-spectrum}
	Assume that the return dynamics satisfies the strict-ascent
	hypotheses of
	Theorem~\ref{thm:morse-spectral-correspondence}.
	At a Morse equilibrium $c_*$ of index
	\[
	\lambda
	=
	\operatorname{ind}_M(c_*),
	\]
	one has
	\begin{equation}
	\boxed{
		\text{sublevel handle index}
		=
		\lambda
		=
		\dim E^s(c_*),
	}
	\label{eq:sublevel-stable}
	\end{equation}
	and
	\begin{equation}
	\boxed{
		\text{superlevel handle index}
		=
		m-\lambda
		=
		\dim E^u(c_*).
	}
	\label{eq:superlevel-unstable}
	\end{equation}
\end{theorem}

\begin{proof}
	The handle indices are $\lambda$ and $m-\lambda$ by
	Theorem~\ref{thm:dual-handle-attachment}.
	Theorem~\ref{thm:morse-spectral-correspondence} gives
	\[
	\dim E^s(c_*)
	=
	\lambda,
	\qquad
	\dim E^u(c_*)
	=
	m-\lambda.
	\]
\end{proof}

Thus the two natural Morse filtrations detect the two complementary
spectral dimensions of the exact return dynamics.

% ---------------------------------------------------------
\subsection{Dynamical orientation of the superlevel filtration}
\label{subsec:superlevel-forward-invariance}
% ---------------------------------------------------------

The sublevel and superlevel filtrations are topologically dual, but
they are not dynamically equivalent.  Since the return map is a
thickness-ascent system, the superlevel filtration is the
dynamically natural one.

\begin{proposition}[Forward invariance of superlevel sets]
	\label{prop:superlevel-forward-invariant}
	Assume
	\begin{equation}
	d(F(c))
	\ge
	d(c)
	\label{eq:nondecreasing-thickness}
	\end{equation}
	for every $c\in\partial C$, with strict inequality away from
	$\operatorname{Crit}(d)$. Then
	\begin{equation}
	\boxed{
		F(U^a)\subset U^a
	}
	\label{eq:superlevel-forward-invariance}
	\end{equation}
	for every $a\in\mathbb R$.
\end{proposition}

\begin{proof}
	If $c\in U^a$, then
	\[
	d(c)\ge a.
	\]
	By \eqref{eq:nondecreasing-thickness},
	\[
	d(F(c))
	\ge
	d(c)
	\ge
	a,
	\]
	so $F(c)\in U^a$.
\end{proof}

Hence along every forward orbit
\[
c_{k+1}=F(c_k),
\]
once a superlevel set is entered it cannot be left.  Forward
iteration is therefore oriented toward regions of larger thickness.

By contrast, the sublevel sets are not in general forward invariant.
Their role is the classical Morse-theoretic one: they record the
topology created by the negative directions of the thickness
Hessian.

% ---------------------------------------------------------
\subsection{Euler jumps and fixed-point indices}
\label{subsec:euler-jumps-fixed-index}
% ---------------------------------------------------------

The superlevel filtration provides a direct topological
interpretation of the local fixed-point index.

At a Morse critical point of index $\lambda$,
Theorem~\ref{thm:fixed-index-morse-index} gives
\[
\operatorname{ind}_{\mathrm{fix}}(F,c_*)
=
(-1)^{m+\lambda}.
\]
Since
\[
(-1)^{m-\lambda}
=
(-1)^{m+\lambda},
\]
comparison with \eqref{eq:superlevel-euler-jump} yields the following
identity.

\begin{theorem}[Local fixed-point index as a superlevel Euler jump]
	\label{thm:index-superlevel-euler-jump}
	Let $c_*$ be a Morse critical point with
	\[
	r_*=d(c_*).
	\]
	Let $a<r_*<b$ be regular values such that $c_*$ is the only critical
	point in
	\[
	d^{-1}([a,b]).
	\]
	Then
	\begin{equation}
	\boxed{
		\operatorname{ind}_{\mathrm{fix}}(F,c_*)
		=
		\chi(U^a)-\chi(U^b).
	}
	\label{eq:index-superlevel-euler-jump}
	\end{equation}
\end{theorem}

The corresponding sublevel relation is
\begin{equation}
\chi(L_b)-\chi(L_a)
=
(-1)^m
\operatorname{ind}_{\mathrm{fix}}(F,c_*).
\label{eq:index-sublevel-euler-jump}
\end{equation}

Thus the local index of the discrete return map is exactly the signed
change in Euler characteristic of the dynamically oriented
superlevel filtration across the corresponding critical level.

Distinct Morse critical points need not have distinct critical
values. Suppose
\[
c_*^{(1)},\ldots,c_*^{(r)}
\]
are all the critical points at a common critical value $r_*$, and
let
\[
\lambda_j
=
\operatorname{ind}_M(c_*^{(j)}).
\]
Assume that
\[
a<r_*<b
\]
contains no other critical value. Then
\begin{equation}
\chi(L_b)-\chi(L_a)
=
\sum_{j=1}^{r}
(-1)^{\lambda_j},
\label{eq:multiple-sublevel-jump}
\end{equation}
and
\begin{equation}
\chi(U^a)-\chi(U^b)
=
\sum_{j=1}^{r}
(-1)^{m-\lambda_j}.
\label{eq:multiple-superlevel-jump}
\end{equation}

Using
Theorem~\ref{thm:fixed-index-morse-index}, we obtain:

\begin{corollary}[Index balance across one critical level]
	\label{cor:index-balance-critical-level}
	Under the preceding assumptions,
	\begin{equation}
	\boxed{
		\chi(U^a)-\chi(U^b)
		=
		\sum_{j=1}^{r}
		\operatorname{ind}_{\mathrm{fix}}
		\bigl(F,c_*^{(j)}\bigr).
	}
	\label{eq:critical-level-index-balance}
	\end{equation}
\end{corollary}

\begin{remark}[Filtration dictionary]
	\label{rem:filtration-dictionary}
	At a Morse equilibrium $c_*$ of index
	\[
	\lambda=\operatorname{ind}_M(c_*),
	\]
	the relations proved above may be summarized as
	\[
	\boxed{
		\begin{array}{c|c}
		\text{thickness-filtration quantity}
		&
		\text{return-dynamical quantity}
		\\ \hline
		\lambda
		&
		\dim E^s(c_*)
		\\[1mm]
		m-\lambda
		&
		\dim E^u(c_*)
		\\[1mm]
		(-1)^{m-\lambda}
		&
		\operatorname{ind}_{\mathrm{fix}}(F,c_*).
		\end{array}}
	\]
	The superlevel filtration is the dynamically natural one because it
	is forward invariant under thickness ascent.
\end{remark}
% =========================================================
% =========================================================
\section{Invariant Manifolds and Morse--Smale Structure}
\label{sec:invariant-manifolds}
% =========================================================

The preceding sections identify the stable and unstable spectral
dimensions of every Morse equilibrium. We now ask when these linear
subspaces generate genuine invariant manifolds and how such manifolds
may organize the global phase space.

A distinction is essential. Spectral hyperbolicity of a fixed point
does not imply that the return map is locally invertible. Stable
manifolds are available under weaker hypotheses, whereas unstable
manifolds and the standard Morse--Smale picture require additional
invertibility assumptions.

Throughout this section, the thickness function is Morse and the
return dynamics is assumed to lie in the strict-ascent regime unless
stated otherwise.

% ---------------------------------------------------------
\subsection{Local invariant manifolds}
\label{subsec:local-invariant-manifolds}
% ---------------------------------------------------------

Let
\[
c_*\in\operatorname{Fix}(F)
=
\operatorname{Crit}(d),
\]
and set
\[
\lambda
=
\operatorname{ind}_M(c_*).
\]
By Theorem~\ref{thm:morse-spectral-correspondence},
\begin{equation}
\dim E^s(c_*)
=
\lambda,
\qquad
\dim E^u(c_*)
=
m-\lambda,
\label{eq:section8-spectral-dimensions}
\end{equation}
and no multiplier of $DF(c_*)$ lies on the unit circle.

The stable-manifold theorem for hyperbolic fixed points of smooth
maps gives the following local result; see, for example,
\cite{mcgeheeSander1996}.

\begin{theorem}[Local stable manifold]
	\label{thm:local-stable-manifold}
	Let $c_*$ be a Morse equilibrium in the strict-ascent regime.
	Then there exist a neighbourhood $U_*$ of $c_*$ and a $C^1$
	embedded disk
	\[
	W^s_{\mathrm{loc}}(c_*)
	\subset U_*
	\]
	such that
	\begin{equation}
	T_{c_*}W^s_{\mathrm{loc}}(c_*)
	=
	E^s(c_*),
	\label{eq:stable-manifold-tangent}
	\end{equation}
	and hence
	\begin{equation}
	\boxed{
		\dim W^s_{\mathrm{loc}}(c_*)
		=
		\operatorname{ind}_M(c_*).
	}
	\label{eq:stable-manifold-dimension}
	\end{equation}
	Points of $W^s_{\mathrm{loc}}(c_*)$ whose forward iterates remain in
	the defining local neighbourhood converge to $c_*$.
\end{theorem}

\begin{remark}[Stable space versus Hessian eigenspace]
	\label{rem:stable-versus-hessian-space}
	The identity
	\[
	\dim E^s(c_*)
	=
	n_-(H_*)
	\]
	is an equality of dimensions, not in general an equality of
	subspaces. The stable eigenspace is determined by $B_*H_*$, whereas
	the negative Hessian space is determined by $H_*$. They coincide
	under additional compatibility conditions, for example when
	$B_*$ and $H_*$ commute. Curvature preconditioning may therefore
	change the eigendirections while preserving their number.
\end{remark}

The unstable side requires more care because the geometric return
map need not admit a unique local inverse.

\begin{proposition}[Local invertibility criterion]
	\label{prop:local-invertibility-criterion}
	At a Morse equilibrium $c_*$,
	\begin{equation}
	DF(c_*)\ \text{is invertible}
	\quad\Longleftrightarrow\quad
	-1\notin\sigma(K_*).
	\label{eq:local-invertibility-K}
	\end{equation}
	Equivalently,
	\begin{equation}
	0\notin\sigma(DF(c_*)).
	\label{eq:local-invertibility-DF}
	\end{equation}
\end{proposition}

\begin{proof}
	The multipliers of $DF(c_*)$ are
	\[
	1+\mu_j,
	\qquad
	\mu_j\in\sigma(K_*).
	\]
	Thus zero is a multiplier precisely when $\mu_j=-1$ for some $j$.
\end{proof}

\begin{theorem}[Local unstable manifold under local invertibility]
	\label{thm:local-unstable-manifold}
	Assume, in addition, that
	\[
	0\notin\sigma(DF(c_*)).
	\]
	Then $F$ is a local $C^1$ diffeomorphism near $c_*$, and there exists
	a local unstable manifold
	\[
	W^u_{\mathrm{loc}}(c_*)
	\]
	satisfying
	\begin{equation}
	T_{c_*}W^u_{\mathrm{loc}}(c_*)
	=
	E^u(c_*),
	\label{eq:unstable-manifold-tangent}
	\end{equation}
	and
	\begin{equation}
	\boxed{
		\dim W^u_{\mathrm{loc}}(c_*)
		=
		m-\operatorname{ind}_M(c_*).
	}
	\label{eq:unstable-manifold-dimension}
	\end{equation}
\end{theorem}

Together with Section~\ref{sec:thickness-filtration}, these formulas
show that, whenever both local invariant manifolds are defined,
their dimensions coincide with the corresponding sublevel and
superlevel handle indices. No identification of the manifolds with
the Morse handles themselves is asserted.

% ---------------------------------------------------------
\subsection{Global stable sets and nonwandering dynamics}
\label{subsec:global-stable-sets}
% ---------------------------------------------------------

The global convergence theory gives a phase-space decomposition
without any Morse--Smale transversality assumption.

For a fixed point $c_*$ define its global stable set by
\begin{equation}
W^s(c_*)
:=
\left\{
c\in\partial C:
F^n(c)\longrightarrow c_*
\right\}.
\label{eq:global-stable-set}
\end{equation}

\begin{theorem}[Decomposition into convergence classes]
	\label{thm:stable-set-decomposition}
	Under the global convergence hypotheses and the Morse assumption,
	\begin{equation}
	\boxed{
		\partial C
		=
		\bigsqcup_{c_*\in\operatorname{Fix}(F)}
		W^s(c_*).
	}
	\label{eq:stable-set-partition}
	\end{equation}
\end{theorem}

\begin{proof}
	Every orbit converges to a unique critical point of $d$ by the global
	convergence theorem. Since
	\[
	\operatorname{Fix}(F)
	=
	\operatorname{Crit}(d),
	\]
	every point belongs to exactly one stable set.
\end{proof}

The decomposition
\eqref{eq:stable-set-partition} is set-theoretic. Without additional
hypotheses, the global stable sets need not form a smooth cell
decomposition.

For a Morse maximum $a$, define its attraction basin by
\begin{equation}
\mathcal B(a)
=
W^s(a).
\label{eq:attraction-basin}
\end{equation}

\begin{proposition}[Openness of attraction basins]
	\label{prop:basin-open}
	If $a$ is a Morse maximum, then $\mathcal B(a)$ is open in
	$\partial C$.
\end{proposition}

\begin{proof}
	By Corollary~\ref{cor:morse-return-stability}, $a$ is a hyperbolic
	attractor. Hence there exists a neighbourhood $U$ of $a$ whose
	points converge to $a$.
	
	If $c\in\mathcal B(a)$, then
	\[
	F^n(c)\in U
	\]
	for some $n\ge0$. By continuity of $F^n$, a neighbourhood of $c$ is
	mapped into $U$ and therefore belongs to $\mathcal B(a)$.
\end{proof}

\begin{proposition}[Location of basin boundaries]
	\label{prop:basin-boundary}
	Assume the hypotheses of Theorem~\ref{thm:stable-set-decomposition}.
	Let $a$ be a Morse maximum. Then
	\begin{equation}
	\partial_{\partial C}\mathcal B(a)
	\subset
	\bigcup_{\substack{
			q\in\operatorname{Fix}(F)\\
			q\ \mathrm{not\ an\ attractor}
	}}
	W^s(q),
	\label{eq:basin-boundary-inclusion}
	\end{equation}
	where the boundary is taken relative to $\partial C$.
\end{proposition}

\begin{proof}
	Let
	\[
	x\in\partial_{\partial C}\mathcal B(a).
	\]
	By Theorem~\ref{thm:stable-set-decomposition}, there is a unique fixed
	point $q$ such that
	\[
	x\in W^s(q).
	\]
	If $q$ were an attractor, then $\mathcal B(q)$ would be open by
	Proposition~\ref{prop:basin-open}. Since attraction basins of
	distinct fixed points are disjoint, a neighbourhood of $x$ would
	either lie in $\mathcal B(a)$ or be disjoint from it, contradicting
	the fact that $x$ lies on its boundary. Hence $q$ is not an
	attractor.
\end{proof}

\begin{remark}[No automatic separatrix theorem]
	\label{rem:no-separatrix-theorem}
	Proposition~\ref{prop:basin-boundary} does not assert that every basin
	boundary is a smooth stable manifold of a saddle, nor that the stable
	manifolds form a regular separatrix complex. Such conclusions require
	additional global invariant-manifold, invertibility, and
	transversality hypotheses.
\end{remark}

The strict Lyapunov structure also determines the nonwandering set.

\begin{definition}[Nonwandering set]
	\label{def:nonwandering}
	A point $c\in\partial C$ is nonwandering if, for every neighbourhood
	$U$ of $c$, there exists $n\ge1$ such that
	\[
	F^n(U)\cap U\neq\varnothing.
	\]
	The set of all nonwandering points is denoted by $\Omega(F)$.
\end{definition}

\begin{proposition}[Nonwandering set in the strict-ascent regime]
	\label{prop:nonwandering-fixed}
	Under strict thickness ascent,
	\begin{equation}
	\boxed{
		\Omega(F)
		=
		\operatorname{Fix}(F).
	}
	\label{eq:nonwandering-fixed}
	\end{equation}
\end{proposition}

\begin{proof}
	Every fixed point is nonwandering.
	
	Conversely, let
	\[
	c\notin\operatorname{Fix}(F).
	\]
	Strict ascent gives
	\[
	\Delta
	:=
	d(F(c))-d(c)
	>
	0.
	\]
	By continuity of $d$ and $F$, there exist a neighbourhood $U$ of
	$c$ and $\eta>0$ such that
	\[
	d(F(y))
	\ge
	d(y)+\eta
	\qquad
	(y\in U),
	\]
	while the oscillation of $d$ on $U$ is smaller than $\eta$.
	After shrinking $U$ if necessary,
	\[
	d(F(y))
	>
	\sup_{z\in U}d(z)
	\qquad
	(y\in U).
	\]
	Since thickness is nondecreasing under all subsequent forward
	iterations,
	\[
	F^n(U)\cap U
	=
	\varnothing
	\qquad
	(n\ge1).
	\]
	Thus $c$ is not nonwandering.
\end{proof}

Hence, in the Morse strict-ascent regime, the nonwandering set is a
finite collection of hyperbolic fixed points.

% ---------------------------------------------------------
\subsection{Conditional Morse--Smale structure}
\label{subsec:conditional-morse-smale}
% ---------------------------------------------------------

The preceding results provide most of the ingredients of a
Morse--Smale system, but neither global invertibility nor
transversality follows from the definition of $\mathcal O_C$ or from
strict ascent. We therefore isolate them as additional hypotheses.

\begin{hypothesis}[Diffeomorphic return regime]
	\label{hyp:diffeomorphic-return}
	The return map
	\[
	F:\partial C\longrightarrow\partial C
	\]
	is a $C^1$ diffeomorphism.
\end{hypothesis}

\begin{hypothesis}[Invariant-manifold transversality]
	\label{hyp:invariant-transversality}
	For every pair of fixed points $p,q$,
	\begin{equation}
	W^u(p)
	\pitchfork
	W^s(q).
	\label{eq:invariant-manifold-transversality}
	\end{equation}
\end{hypothesis}

These hypotheses are deliberately separated from the basic geometric
return-map assumptions.

\begin{theorem}[Conditional Morse--Smale property]
	\label{thm:conditional-morse-smale}
	Assume that:
	\begin{enumerate}
		\item $d$ is Morse;
		\item the return dynamics lies in the strict global ascent regime;
		\item Hypothesis~\ref{hyp:diffeomorphic-return} holds;
		\item Hypothesis~\ref{hyp:invariant-transversality} holds.
	\end{enumerate}
	Then $F$ is a Morse--Smale diffeomorphism: its nonwandering set is
	finite and hyperbolic, and its stable and unstable manifolds
	intersect transversely.
\end{theorem}

\begin{proof}
	By Proposition~\ref{prop:nonwandering-fixed},
	\[
	\Omega(F)
	=
	\operatorname{Fix}(F).
	\]
	The Morse hypothesis makes this set finite, and
	Theorem~\ref{thm:morse-spectral-correspondence} shows that every fixed
	point is hyperbolic. Transversality is precisely
	Hypothesis~\ref{hyp:invariant-transversality}. Under
	Hypothesis~\ref{hyp:diffeomorphic-return}, these are the standard
	Morse--Smale conditions.
\end{proof}

The Lyapunov function imposes a strict orientation on all
heteroclinic connections.

\begin{proposition}[Thickness ordering of heteroclinic connections]
	\label{prop:heteroclinic-thickness-order}
	Assume the hypotheses of
	Theorem~\ref{thm:conditional-morse-smale}. If
	\[
	W^u(p)\cap W^s(q)\neq\varnothing,
	\qquad
	p\neq q,
	\]
	then
	\begin{equation}
	\boxed{
		d(p)<d(q).
	}
	\label{eq:heteroclinic-thickness-order}
	\end{equation}
\end{proposition}

\begin{proof}
	Choose
	\[
	x\in W^u(p)\cap W^s(q).
	\]
	Since $F$ is a diffeomorphism,
	\[
	F^{-n}(x)\longrightarrow p,
	\qquad
	F^n(x)\longrightarrow q.
	\]
	The orbit through $x$ is nonstationary because $p\neq q$, and strict
	ascent gives
	\[
	d(F^{k+1}(x))
	>
	d(F^k(x))
	\qquad
	(k\in\mathbb Z).
	\]
	Passing to the limits $k\to-\infty$ and $k\to+\infty$ yields
	\[
	d(p)<d(q).
	\]
\end{proof}

\begin{corollary}[Acyclic heteroclinic graph]
	\label{cor:acyclic-heteroclinic-graph}
	Form a directed graph whose vertices are the fixed points and whose
	arrows
	\[
	p\longrightarrow q
	\]
	correspond to nonempty intersections
	\[
	W^u(p)\cap W^s(q).
	\]
	Then this directed graph is acyclic.
\end{corollary}

\begin{proof}
	Along every directed edge,
	Proposition~\ref{prop:heteroclinic-thickness-order} gives a strict
	increase of $d$. A directed cycle would therefore imply
	\[
	d(p_1)
	<
	d(p_2)
	<
	\cdots
	<
	d(p_r)
	<
	d(p_1),
	\]
	which is impossible.
\end{proof}

Transversality also restricts the Morse indices of heteroclinically
connected equilibria.

\begin{proposition}[Dimension of a heteroclinic intersection]
	\label{prop:heteroclinic-dimension}
	Assume the hypotheses of
	Theorem~\ref{thm:conditional-morse-smale}.
	If
	\[
	W^u(p)\pitchfork W^s(q)
	\]
	and the intersection is nonempty, then
	\begin{equation}
	\dim\bigl(W^u(p)\cap W^s(q)\bigr)
	=
	\operatorname{ind}_M(q)
	-
	\operatorname{ind}_M(p).
	\label{eq:heteroclinic-dimension}
	\end{equation}
	Consequently,
	\begin{equation}
	\operatorname{ind}_M(p)
	\le
	\operatorname{ind}_M(q).
	\label{eq:heteroclinic-index-order}
	\end{equation}
\end{proposition}

\begin{proof}
	By the Morse-index/spectral-dimension correspondence,
	\[
	\dim W^u(p)
	=
	m-\operatorname{ind}_M(p),
	\]
	while
	\[
	\dim W^s(q)
	=
	\operatorname{ind}_M(q).
	\]
	A transverse intersection in the $m$-dimensional manifold
	$\partial C$ therefore has dimension
	\[
	\bigl(m-\operatorname{ind}_M(p)\bigr)
	+
	\operatorname{ind}_M(q)
	-
	m,
	\]
	which is
	\[
	\operatorname{ind}_M(q)-\operatorname{ind}_M(p).
	\]
	A nonempty transverse intersection cannot have negative dimension.
\end{proof}

Combining the two ordering principles gives
\begin{equation}
\boxed{
	p\longrightarrow q
	\quad\Longrightarrow\quad
	\begin{cases}
	d(p)<d(q),\\[1mm]
	\operatorname{ind}_M(p)
	\le
	\operatorname{ind}_M(q).
	\end{cases}}
\label{eq:heteroclinic-double-order}
\end{equation}

\begin{remark}[Comparison with gradient ascent]
	\label{rem:gradient-ascent-comparison}
	The ordering
	\eqref{eq:heteroclinic-double-order}
	is consistent with a gradient-ascent picture: heteroclinic
	connections run toward larger thickness, and the stable dimension
	cannot decrease. The return map is not assumed to be the time-one
	map of a gradient flow; these conclusions follow instead from the
	strict Lyapunov structure and the Morse-index/spectral
	correspondence.
\end{remark}

\begin{remark}[No automatic Morse complex]
	\label{rem:no-automatic-morse-complex}
	For Morse--Smale gradient-like flows, invariant manifolds are closely
	related to CW decompositions of the underlying manifold; see, for
	example, \cite{abbondandoloMajer2023}. In the present geometrically
	generated discrete system, equality of dimensions between invariant
	manifolds and Morse handles does not by itself identify the two.
	Establishing conditions under which the invariant manifolds of $F$
	realize the handle decomposition of the thickness landscape remains
	a separate problem.
\end{remark}

\begin{remark}[Role of the diffeomorphism hypothesis]
	\label{rem:role-diffeomorphism}
	The geometric construction provides a forward return map but does not
	in general provide a unique backward return. Consequently, the
	stable-set decomposition, basin results, and nonwandering-set
	description above do not require global invertibility. The global
	unstable-manifold and Morse--Smale conclusions are stated only under
	Hypothesis~\ref{hyp:diffeomorphic-return}.
\end{remark}
% =========================================================
% =========================================================
\section{Examples}
\label{sec:examples}
% =========================================================

We now illustrate the preceding theory by a collection of examples
chosen to emphasize complementary aspects of the
geometry--dynamics--topology correspondence.

The first example is one-dimensional and makes the spectral threshold
$\mu=-2$ completely explicit.  The second gives a genuinely
nontrivial Morse landscape on $\mathbb S^2$ containing repellers,
saddles, and attractors simultaneously, and illustrates the Morse
polynomial, fixed-point indices, curvature comparison, and filtration
jumps in a single model.  The third shows how spherical cores reduce
the curvature preconditioner to a scalar operator in arbitrary
dimension.  Finally, a local two-dimensional example illustrates the
noncommuting situation in which the core and outer-boundary principal
directions are not aligned.

The Morse-theoretic calculations below concern the thickness
functions themselves. Whenever a dynamical classification is invoked,
we additionally assume that the corresponding radial graph defines an
admissible domain in $\mathcal O_C$ and that the strict-ascent
hypotheses of Section~\ref{sec:framework} hold. These geometric
requirements are not consequences of positivity of the displayed
thickness functions alone.

% ---------------------------------------------------------
\subsection{A periodic cosine landscape on the circle}
\label{subsec:circle-example}
% ---------------------------------------------------------

Let
\[
C=\overline{B_1(0)}
\subset\mathbb R^2,
\qquad
\partial C=\mathbb S^1.
\]
Writing
\[
c(\theta)
=
(\cos\theta,\sin\theta),
\qquad
\theta\in\mathbb R/2\pi\mathbb Z,
\]
the shape operator in the convention of
\eqref{eq:shape-operator} is
\[
S=-I.
\]
Hence
\begin{equation}
B(\theta)
=
\frac{d(\theta)}
{1+d(\theta)}.
\label{eq:circle-B}
\end{equation}

Since $\theta$ is the arc-length coordinate on the unit circle,
\[
H(\theta)=d''(\theta),
\]
and therefore
\begin{equation}
K(\theta)
=
\frac{d(\theta)}
{1+d(\theta)}
d''(\theta).
\label{eq:circle-K}
\end{equation}

At a critical point $\theta_*$, the multiplier of the return map is
\begin{equation}
F'(\theta_*)
=
1+
\frac{d(\theta_*)}
{1+d(\theta_*)}
d''(\theta_*).
\label{eq:circle-multiplier}
\end{equation}

Consider the family
\begin{equation}
d(\theta)
=
d_0+\varepsilon\cos(\ell\theta),
\qquad
d_0>\varepsilon>0,
\qquad
\ell\in\mathbb N.
\label{eq:cosine-thickness}
\end{equation}
The critical points satisfy
\[
\sin(\ell\theta)=0,
\]
so there are exactly $2\ell$ of them.

The maxima are
\begin{equation}
\theta_k^+
=
\frac{2k\pi}{\ell},
\qquad
k=0,\ldots,\ell-1,
\label{eq:circle-maxima}
\end{equation}
with
\[
d(\theta_k^+)=d_0+\varepsilon,
\qquad
d''(\theta_k^+)
=
-\varepsilon\ell^2,
\]
whereas the minima are
\begin{equation}
\theta_k^-
=
\frac{(2k+1)\pi}{\ell},
\qquad
k=0,\ldots,\ell-1,
\label{eq:circle-minima}
\end{equation}
with
\[
d(\theta_k^-)=d_0-\varepsilon,
\qquad
d''(\theta_k^-)
=
+\varepsilon\ell^2.
\]

Thus
\begin{equation}
m_0=m_1=\ell,
\label{eq:circle-cosine-counts}
\end{equation}
and the Morse polynomial is
\begin{equation}
M_d(t)
=
\ell(1+t).
\label{eq:circle-cosine-morse-polynomial}
\end{equation}
In the notation of Section~\ref{sec:morse-inequalities},
\[
Q(t)=\ell-1,
\]
so that
\[
\#\operatorname{Crit}(d)
=
2+2Q(1)
=
2\ell.
\]

Under strict ascent, the $\ell$ minima are repellers and the $\ell$
maxima are attractors. Hence
\begin{equation}
E_F(t)
=
\ell(1+t).
\label{eq:circle-equilibrium-polynomial}
\end{equation}

\begin{proposition}[Multipliers of the cosine landscape]
	\label{prop:circle-cosine-multipliers}
	For the thickness profile
	\eqref{eq:cosine-thickness}, the preconditioned Hessian eigenvalue at
	a maximum is
	\begin{equation}
	\mu_+
	=
	-\varepsilon\ell^2
	\frac{d_0+\varepsilon}
	{1+d_0+\varepsilon},
	\label{eq:circle-mu-max}
	\end{equation}
	whereas at a minimum it is
	\begin{equation}
	\mu_-
	=
	+\varepsilon\ell^2
	\frac{d_0-\varepsilon}
	{1+d_0-\varepsilon}.
	\label{eq:circle-mu-min}
	\end{equation}
	The corresponding multipliers are
	\begin{align}
	\Lambda_+
	&=
	1-
	\varepsilon\ell^2
	\frac{d_0+\varepsilon}
	{1+d_0+\varepsilon},
	\label{eq:circle-lambda-max}
	\\
	\Lambda_-
	&=
	1+
	\varepsilon\ell^2
	\frac{d_0-\varepsilon}
	{1+d_0-\varepsilon}.
	\label{eq:circle-lambda-min}
	\end{align}
\end{proposition}

At a minimum,
\[
\Lambda_->1,
\]
as expected for a repelling Morse equilibrium.

At a maximum, local attraction is equivalent to
\begin{equation}
-2<\mu_+<0,
\label{eq:circle-max-attraction}
\end{equation}
or
\begin{equation}
\varepsilon\ell^2
\frac{d_0+\varepsilon}
{1+d_0+\varepsilon}
<2.
\label{eq:circle-spectral-threshold}
\end{equation}
The limiting equality
\begin{equation}
\varepsilon\ell^2
\frac{d_0+\varepsilon}
{1+d_0+\varepsilon}
=
2
\label{eq:circle-minus-two-threshold}
\end{equation}
corresponds to
\[
\Lambda_+=-1.
\]
Thus it marks the local spectral boundary of stability of the Morse
maximum. No bifurcation conclusion is inferred from this equality
alone.

Notice that
\eqref{eq:circle-spectral-threshold} is a local condition at the
maxima. It should not be confused with the uniform global spectral
sufficient condition of Section~\ref{sec:framework}, which must hold
at every point of the phase space together with the independent
small-slope and return-transversality assumptions.

The fixed-point indices are
\[
-1
\]
at the $\ell$ repellers and
\[
+1
\]
at the $\ell$ attractors. Consequently,
\[
-\ell+\ell
=
0
=
\chi(\mathbb S^1),
\]
recovering simultaneously the Morse and fixed-point-index balances.

% ---------------------------------------------------------
\subsection{A six-equilibrium Morse landscape on the two-sphere}
\label{subsec:six-equilibrium-sphere}
% ---------------------------------------------------------

We next consider an example in which all three Morse types occur.

Let
\[
C=\overline{B_1(0)}
\subset\mathbb R^3,
\qquad
\partial C=\mathbb S^2,
\]
and define
\begin{equation}
d(x,y,z)
=
d_0+\varepsilon
\left(
ax^2+by^2+cz^2
\right),
\qquad
0<a<b<c,
\label{eq:quadratic-sphere-landscape}
\end{equation}
where $d_0>0$ and $\varepsilon>0$.

The critical points of the restriction of the quadratic form to
$\mathbb S^2$ are
\[
\pm e_1,
\qquad
\pm e_2,
\qquad
\pm e_3.
\]
Their critical values are
\begin{equation}
d_1=d_0+\varepsilon a,
\qquad
d_2=d_0+\varepsilon b,
\qquad
d_3=d_0+\varepsilon c,
\label{eq:quadratic-critical-values}
\end{equation}
with
\[
d_1<d_2<d_3.
\]

Let
\[
A=\operatorname{diag}(a,b,c).
\]
At an eigenvector $e_i$ of $A$, the intrinsic Hessian of the
restriction of
\[
\varepsilon\,x^TAx
\]
to the sphere is
\[
2\varepsilon
(A-a_iI)
\]
restricted to the tangent plane $T_{e_i}\mathbb S^2$. Hence
\begin{align}
\sigma(H_{\pm e_1})
&=
\left\{
2\varepsilon(b-a),
2\varepsilon(c-a)
\right\},
\label{eq:H-e1}
\\
\sigma(H_{\pm e_2})
&=
\left\{
2\varepsilon(a-b),
2\varepsilon(c-b)
\right\},
\label{eq:H-e2}
\\
\sigma(H_{\pm e_3})
&=
\left\{
2\varepsilon(a-c),
2\varepsilon(b-c)
\right\}.
\label{eq:H-e3}
\end{align}

It follows that
\begin{equation}
\operatorname{ind}_M(\pm e_1)=0,
\qquad
\operatorname{ind}_M(\pm e_2)=1,
\qquad
\operatorname{ind}_M(\pm e_3)=2.
\label{eq:quadratic-sphere-indices}
\end{equation}
Thus $\pm e_1$ are minima, $\pm e_2$ are saddles, and $\pm e_3$
are maxima.

Consequently,
\begin{equation}
m_0=m_1=m_2=2,
\label{eq:six-equilibrium-counts}
\end{equation}
and
\begin{equation}
\begin{aligned}
M_d(t)
&=
2+2t+2t^2
\\
&=
1+t^2+(1+t)(1+t).
\end{aligned}
\label{eq:six-equilibrium-morse-polynomial}
\end{equation}
Therefore
\begin{equation}
Q(t)=1+t,
\label{eq:six-equilibrium-Q}
\end{equation}
and the general parity formula gives
\[
\#\operatorname{Crit}(d)
=
2+2Q(1)
=
6.
\]

This example realizes a symmetric nonperfect Morse configuration on $\mathbb S^2$ in which all three Morse types occur.

\subsubsection*{Return spectrum}

Because the core is the unit sphere,
\[
S=-I,
\qquad
B_*=
\frac{d_*}{1+d_*}I.
\]
Set
\begin{equation}
\beta_i
=
\frac{d_i}{1+d_i},
\qquad
i=1,2,3.
\label{eq:quadratic-beta}
\end{equation}
Then
\begin{align}
\sigma(K_{\pm e_1})
&=
\left\{
2\varepsilon\beta_1(b-a),
2\varepsilon\beta_1(c-a)
\right\},
\label{eq:K-e1}
\\
\sigma(K_{\pm e_2})
&=
\left\{
2\varepsilon\beta_2(a-b),
2\varepsilon\beta_2(c-b)
\right\},
\label{eq:K-e2}
\\
\sigma(K_{\pm e_3})
&=
\left\{
2\varepsilon\beta_3(a-c),
2\varepsilon\beta_3(b-c)
\right\}.
\label{eq:K-e3}
\end{align}

Under the strict-ascent hypotheses, all negative eigenvalues in
\eqref{eq:K-e2}--\eqref{eq:K-e3} lie in $(-2,0)$. Hence
\[
\dim E^s(\pm e_1)=0,
\qquad
\dim E^s(\pm e_2)=1,
\qquad
\dim E^s(\pm e_3)=2.
\]
The corresponding return dynamics therefore consists of
\begin{equation}
\boxed{
	2\ \text{repellers},
	\qquad
	2\ \text{saddles},
	\qquad
	2\ \text{attractors}.
}
\label{eq:six-equilibrium-dynamics}
\end{equation}
The equilibrium polynomial is
\begin{equation}
E_F(t)
=
2+2t+2t^2
=
M_d(t).
\label{eq:six-equilibrium-EF}
\end{equation}

\subsubsection*{Curvature interpretation}

For the unit spherical core,
\[
\kappa^C=1,
\]
and the principal curvature of the parallel surface at distance
$d_i$ is
\begin{equation}
\kappa_i^{\mathrm{par}}
=
\frac{1}{1+d_i}.
\label{eq:quadratic-parallel-curvature}
\end{equation}

At a critical point, the spherical-core symmetry implies that the
Hessian eigendirections are also principal directions of the outer
boundary.  From
\eqref{eq:H-parallel-curvature}, an eigenvalue $\eta$ of $H_*$ and
the corresponding outer principal curvature satisfy
\begin{equation}
\kappa^\Omega
=
\frac{1}{1+d_*}
-
\frac{\eta}{(1+d_*)^2}.
\label{eq:outer-curvature-from-H}
\end{equation}

Thus:

\begin{itemize}
	\item
	at $\pm e_1$, both Hessian eigenvalues are positive, so
	\[
	\kappa^\Omega
	<
	\kappa^{\mathrm{par}}
	\]
	in both tangent directions;
	
	\item
	at $\pm e_2$, one Hessian eigenvalue is negative and one positive,
	so
	\[
	\kappa^\Omega
	>
	\kappa^{\mathrm{par}}
	\]
	in exactly one tangent direction;
	
	\item
	at $\pm e_3$, both Hessian eigenvalues are negative, so
	\[
	\kappa^\Omega
	>
	\kappa^{\mathrm{par}}
	\]
	in both tangent directions.
\end{itemize}

Hence the example realizes explicitly the curvature-count formula
\[
\operatorname{ind}_M(c_*)
=
\#
\left\{
i:
\kappa_i^\Omega>\kappa_i^{\mathrm{par}}
\right\}.
\]

\subsubsection*{Fixed-point indices and filtration}

Since $m=2$,
\[
\operatorname{ind}_{\mathrm{fix}}(F,c_*)
=
(-1)^{\operatorname{ind}_M(c_*)}.
\]
Therefore
\begin{align}
\operatorname{ind}_{\mathrm{fix}}(F,\pm e_1)
&=
+1,
\\
\operatorname{ind}_{\mathrm{fix}}(F,\pm e_2)
&=
-1,
\\
\operatorname{ind}_{\mathrm{fix}}(F,\pm e_3)
&=
+1.
\end{align}
Their global sum is
\begin{equation}
2-2+2
=
2
=
\chi(\mathbb S^2).
\label{eq:six-equilibrium-index-balance}
\end{equation}

The example also illustrates the case of several critical points at
the same critical level. As the sublevel threshold increases,

\[
\begin{array}{c|c|c}
\text{critical level}
&
\text{handle attachment}
&
\text{Euler jump}
\\ \hline
d_1
&
2\text{ handles of index }0
&
+2
\\
d_2
&
2\text{ handles of index }1
&
-2
\\
d_3
&
2\text{ handles of index }2
&
+2.
\end{array}
\]

The superlevel filtration encounters these critical levels in the
opposite direction. Its signed Euler jumps coincide with the sums of
the local fixed-point indices at the corresponding levels, as in
Corollary~\ref{cor:index-balance-critical-level}.

\begin{figure}[htbp]
	\centering
	\begin{tikzpicture}[
	scale=2.0,
	>=Latex,
	attractor/.style={circle,fill=black,inner sep=2.2pt},
	repeller/.style={circle,draw=black,fill=white,inner sep=2.4pt},
	saddle/.style={rectangle,draw=black,fill=white,inner sep=2.0pt}
	]
	
	% Schematic sphere
	\draw[thick] (0,0) circle (1.5);
	
	% Visible and hidden equator
	\draw[thick]
	(-1.5,0)
	arc[start angle=180,end angle=360,
	x radius=1.5,y radius=0.42];
	
	\draw[dashed]
	(1.5,0)
	arc[start angle=0,end angle=180,
	x radius=1.5,y radius=0.42];
	
	% Critical points
	\node[attractor,label=right:$A_{+}$] at (0,1.5) {};
	\node[attractor,label=right:$A_{-}$] at (0,-1.5) {};
	
	\node[repeller,label=below:$R_{+}$] at (1.5,0) {};
	\node[repeller,label=below:$R_{-}$] at (-1.5,0) {};
	
	\node[saddle,label=above right:$S_{+}$] at (0.92,0.55) {};
	\node[saddle,label=below left:$S_{-}$] at (-0.92,-0.55) {};
	
	% Schematic ascent directions
	\draw[-{Latex[length=2mm]},thick]
	(1.36,0.12)
	.. controls (1.22,0.31) and (1.08,0.43) ..
	(0.99,0.51);
	
	\draw[-{Latex[length=2mm]},thick]
	(0.83,0.67)
	.. controls (0.58,1.03) and (0.30,1.28) ..
	(0.08,1.42);
	
	\draw[-{Latex[length=2mm]},thick]
	(-1.36,-0.12)
	.. controls (-1.22,-0.31) and (-1.08,-0.43) ..
	(-0.99,-0.51);
	
	\draw[-{Latex[length=2mm]},thick]
	(-0.83,-0.67)
	.. controls (-0.58,-1.03) and (-0.30,-1.28) ..
	(-0.08,-1.42);
	
	% Legend
	\node[repeller] at (-2.28,1.10) {};
	\node[anchor=west] at (-2.18,1.10)
	{\small minimum / repeller};
	
	\node[saddle] at (-2.28,0.80) {};
	\node[anchor=west] at (-2.18,0.80)
	{\small saddle};
	
	\node[attractor] at (-2.28,0.50) {};
	\node[anchor=west] at (-2.18,0.50)
	{\small maximum / attractor};
	
	\end{tikzpicture}
	
	\caption{
		Schematic Morse organization of the quadratic thickness landscape
		\eqref{eq:quadratic-sphere-landscape} on $\mathbb S^2$.
		The two minima are repellers, the two index-one critical points are
		saddles, and the two maxima are attractors in the strict-ascent
		regime. The arrows indicate only the orientation toward increasing
		thickness; they are not asserted to be exact heteroclinic trajectories
		of the return map.
	}
	\label{fig:quadratic-sphere-morse}
\end{figure}

Figure~\ref{fig:quadratic-sphere-morse} should therefore be read as a
Morse and Lyapunov schematic rather than as a complete phase portrait.
Exact heteroclinic connections require the additional hypotheses of
Section~\ref{sec:invariant-manifolds}.

% ---------------------------------------------------------
\subsection{Spherical-core reduction and a perfect landscape}
\label{subsec:spherical-core-perfect}
% ---------------------------------------------------------

The preceding two examples exploit a general simplification that
occurs whenever the convex core is spherical.

Let
\[
C=\overline{B_R(0)}
\subset\mathbb R^{m+1},
\qquad
R>0,
\]
so that
\[
\partial C=\mathbb S_R^m.
\]
The shape operator is
\begin{equation}
S
=
-\frac1R I.
\label{eq:sphere-radius-R-shape}
\end{equation}
Hence
\begin{equation}
\boxed{
	B_c
	=
	\frac{R\,d(c)}
	{R+d(c)}I.
}
\label{eq:sphere-radius-R-B}
\end{equation}

\begin{proposition}[Spectral reduction on a spherical core]
	\label{prop:spherical-core-spectral-reduction}
	Let $c_*$ be a Morse critical point of a thickness function on
	$\mathbb S_R^m$. Let
	\[
	\eta_1,\ldots,\eta_m
	\]
	be the eigenvalues of
	\[
	H_*
	=
	\operatorname{Hess}_{\mathbb S_R^m}d(c_*).
	\]
	Then the preconditioned Hessian has eigenvalues
	\begin{equation}
	\mu_j
	=
	\frac{R\,d_*}
	{R+d_*}
	\eta_j,
	\qquad
	d_*=d(c_*),
	\label{eq:spherical-mu-general}
	\end{equation}
	and the corresponding return multipliers are
	\begin{equation}
	\boxed{
		\Lambda_j
		=
		1+
		\frac{R\,d_*}
		{R+d_*}
		\eta_j.
	}
	\label{eq:spherical-multiplier-general}
	\end{equation}
\end{proposition}

\begin{proof}
	Equation~\eqref{eq:sphere-radius-R-B} gives
	\[
	B_*
	=
	\frac{R\,d_*}
	{R+d_*}I.
	\]
	Therefore
	\[
	K_*
	=
	B_*^{1/2}H_*B_*^{1/2}
	=
	\frac{R\,d_*}
	{R+d_*}H_*,
	\]
	and the result follows.
\end{proof}

In particular, on a spherical core the preconditioner does not rotate
the Hessian eigendirections: it changes only their spectral scale.
This makes spherical cores a transparent testing ground for the
general theory.

For a negative Hessian eigenvalue $\eta_j<0$, the local stability
condition is
\begin{equation}
-2
<
\frac{R\,d_*}
{R+d_*}
\eta_j
<
0,
\label{eq:spherical-negative-condition}
\end{equation}
or equivalently
\begin{equation}
\eta_j
>
-\frac{2(R+d_*)}
{R\,d_*}.
\label{eq:spherical-hessian-threshold}
\end{equation}

We now specialize to the height landscape
\begin{equation}
d(R\omega)
=
d_0+\varepsilon\omega_{m+1},
\qquad
\omega\in\mathbb S^m,
\qquad
d_0>\varepsilon>0.
\label{eq:sphere-R-height-thickness}
\end{equation}
It has exactly two critical points,
\[
p_+=Re_{m+1},
\qquad
p_-=-Re_{m+1},
\]
where $p_+$ is the unique maximum and $p_-$ the unique minimum.

Since
\begin{equation}
\operatorname{Hess}_{\mathbb S_R^m}
(\omega_{m+1})
=
-\frac{\omega_{m+1}}{R^2}
g_{\mathbb S_R^m},
\label{eq:sphere-R-height-hessian}
\end{equation}
one obtains
\[
H_+
=
-\frac{\varepsilon}{R^2}I,
\qquad
H_-
=
+\frac{\varepsilon}{R^2}I.
\]
Thus
\begin{equation}
\operatorname{ind}_M(p_+)=m,
\qquad
\operatorname{ind}_M(p_-)=0,
\label{eq:sphere-height-indices}
\end{equation}
and the landscape is perfect:
\begin{equation}
M_d(t)
=
1+t^m.
\label{eq:sphere-R-perfect-M}
\end{equation}

Writing
\[
d_+=d_0+\varepsilon,
\qquad
d_-=d_0-\varepsilon,
\]
the preconditioned eigenvalues are
\begin{align}
\mu_+
&=
-
\frac{\varepsilon d_+}
{R(R+d_+)},
\label{eq:sphere-R-mu-plus}
\\
\mu_-
&=
+
\frac{\varepsilon d_-}
{R(R+d_-)}.
\label{eq:sphere-R-mu-minus}
\end{align}

Under the strict-ascent hypotheses, $p_-$ is therefore a repeller and
$p_+$ an attractor, with no saddle equilibria. Moreover,
\begin{equation}
E_F(t)
=
M_d(t)
=
1+t^m.
\label{eq:sphere-R-perfect-E}
\end{equation}

Their fixed-point indices are
\begin{equation}
\operatorname{ind}_{\mathrm{fix}}(F,p_-)
=
(-1)^m,
\qquad
\operatorname{ind}_{\mathrm{fix}}(F,p_+)
=
1,
\label{eq:sphere-perfect-indices}
\end{equation}
and hence
\[
(-1)^m+1
=
\chi(\mathbb S^m).
\]

The filtration is equally simple. The sublevel filtration begins with
a $0$-handle at the repelling minimum and is completed by an
$m$-handle at the attracting maximum. The superlevel filtration
reverses these roles: it begins at the maximum with a $0$-handle and
is completed at the minimum by an $m$-handle.

\begin{remark}[Global phase portrait in the diffeomorphic regime]
	\label{rem:perfect-global-phase-portrait}
	Under global convergence,
	\[
	\mathbb S_R^m
	=
	W^s(p_-)
	\sqcup
	W^s(p_+).
	\]
	Without injectivity, the stable set of the repeller may contain
	points that reach $p_-$ after finitely many iterations. If $F$ is a
	diffeomorphism, however,
	\[
	W^s(p_-)=\{p_-\},
	\]
	and therefore
	\begin{equation}
	F^n(\omega)
	\longrightarrow
	p_+
	\qquad
	\text{for every }\omega\neq p_-.
	\label{eq:perfect-global-attraction}
	\end{equation}
\end{remark}

\begin{remark}[The radial graph need not be a round sphere]
	\label{rem:height-not-translated-sphere}
	The thickness profile
	\eqref{eq:sphere-R-height-thickness} does not in general describe a
	translated or concentric round sphere. The calculations use the
	spherical geometry of the convex core together with the exact
	return-map formulas; no radiality assumption is imposed on the inward
	normal field of the outer boundary.
\end{remark}

% ---------------------------------------------------------
\subsection{A local noncommuting curvature example}
\label{subsec:noncommuting-example}
% ---------------------------------------------------------

The spherical examples above have
\[
B_*=\beta_*I,
\]
and therefore do not display one genuinely geometric feature of the
general theory: the preconditioner and thickness Hessian need not
commute.

We now give a local two-dimensional example illustrating this
phenomenon.  No global realization of the displayed local data is
asserted; the purpose is to exhibit explicitly the algebraic
mechanism described in
Remark~\ref{rem:noncommuting-curvature}.

Let $m=2$, and choose an orthonormal basis of
$T_{c_*}\partial C$ consisting of principal directions of the core.
Write
\begin{equation}
S_*
=
\begin{pmatrix}
-\kappa_1 & 0\\
0 & -\kappa_2
\end{pmatrix},
\qquad
\kappa_1,\kappa_2\ge0.
\label{eq:noncommuting-S}
\end{equation}
Then
\[
B_*
=
\begin{pmatrix}
b_1&0\\
0&b_2
\end{pmatrix},
\qquad
b_i
=
\frac{d_*}
{1+d_*\kappa_i}.
\]
Let
\begin{equation}
H_*
=
\begin{pmatrix}
\eta_1 & \eta_{12}\\
\eta_{12} & \eta_2
\end{pmatrix},
\qquad
\eta_{12}\neq0.
\label{eq:noncommuting-H}
\end{equation}
If
\[
b_1\neq b_2,
\]
then
\begin{equation}
B_*H_*
=
\begin{pmatrix}
b_1\eta_1 & b_1\eta_{12}\\
b_2\eta_{12} & b_2\eta_2
\end{pmatrix}
\label{eq:noncommuting-BH}
\end{equation}
is generally not self-adjoint.

The self-adjoint preconditioned Hessian is instead
\begin{equation}
K_*
=
\begin{pmatrix}
b_1\eta_1
&
\sqrt{b_1b_2}\,\eta_{12}
\\
\sqrt{b_1b_2}\,\eta_{12}
&
b_2\eta_2
\end{pmatrix}.
\label{eq:noncommuting-K}
\end{equation}
Its eigenvalues are
\begin{equation}
\mu_\pm
=
\frac12
\left[
b_1\eta_1+b_2\eta_2
\pm
\sqrt{
	(b_1\eta_1-b_2\eta_2)^2
	+
	4b_1b_2\eta_{12}^2
}
\right].
\label{eq:noncommuting-mu}
\end{equation}
The return multipliers are
\[
\Lambda_\pm
=
1+\mu_\pm.
\]

Although $B_*H_*$ need not be self-adjoint, it is similar to $K_*$,
and therefore has the same two real eigenvalues. At the same time,
$K_*$ is congruent to $H_*$ through $B_*^{1/2}$, so the two matrices
have the same inertia.

For example, take
\[
d_*=1,
\qquad
\kappa_1=1,
\qquad
\kappa_2=2,
\]
and
\begin{equation}
H_*
=
\begin{pmatrix}
1&1\\
1&-1
\end{pmatrix}.
\label{eq:noncommuting-numerical-H}
\end{equation}
Then
\[
B_*
=
\begin{pmatrix}
1/2&0\\
0&1/3
\end{pmatrix},
\]
whereas
\begin{equation}
B_*H_*
=
\begin{pmatrix}
1/2&1/2\\
1/3&-1/3
\end{pmatrix}
\label{eq:noncommuting-numerical-BH}
\end{equation}
is not symmetric. The self-adjoint operator is
\begin{equation}
K_*
=
\begin{pmatrix}
1/2&1/\sqrt6\\
1/\sqrt6&-1/3
\end{pmatrix},
\label{eq:noncommuting-numerical-K}
\end{equation}
with eigenvalues
\begin{equation}
\mu_-=-\frac12,
\qquad
\mu_+=\frac23.
\label{eq:noncommuting-numerical-mu}
\end{equation}
Thus the return multipliers are
\begin{equation}
\Lambda_-=\frac12,
\qquad
\Lambda_+=\frac53.
\label{eq:noncommuting-numerical-lambda}
\end{equation}

The Hessian has one negative and one positive eigenvalue, so the
critical point has Morse index one. The same inertia is seen in
$K_*$, while the return linearization has exactly one stable and one
unstable multiplier. This gives a local saddle configuration
compatible with the strict-ascent spectral regime.

The off-diagonal term in \eqref{eq:noncommuting-H} also has a direct
curvature interpretation. Since
\begin{equation}
H_*
=
L_*S_*^\Omega L_*-L_*S_*,
\label{eq:noncommuting-curvature-recall}
\end{equation}
and $L_*$ is diagonal in the chosen core-principal basis,
\begin{equation}
(H_*)_{12}
=
(L_*)_{11}(L_*)_{22}
(S_*^\Omega)_{12}.
\label{eq:noncommuting-offdiagonal}
\end{equation}
Hence
\[
\eta_{12}\neq0
\]
is precisely the local signature that the outer-boundary shape
operator does not preserve the principal frame of the core.

This example demonstrates why the invariant operator formula
\[
K_*
=
d_*
\left(
L_*^{1/2}S_*^\Omega L_*^{1/2}-S_*
\right)
\]
is needed beyond the aligned situation: curvature still determines
the return spectrum exactly, but no pairwise comparison of
corresponding principal curvatures is available.

\begin{remark}[What the examples show]
	\label{rem:examples-summary}
	The four examples isolate different parts of the general theory.
	
	The circle makes the multiplier threshold $\mu=-2$ explicit.
	The quadratic landscape on $\mathbb S^2$ simultaneously realizes
	repelling, saddle, and attracting Morse equilibria and displays the
	Morse polynomial, curvature count, fixed-point-index balance, and
	filtration jumps. The arbitrary spherical core shows that curvature
	preconditioning becomes scalar and hence preserves eigendirections.
	Finally, the noncommuting example shows that for a general convex
	core the natural spectral object is the self-adjoint operator $K_*$,
	not the generally non-self-adjoint product $B_*H_*$.
\end{remark}
% =========================================================
% =========================================================
\section{Beyond Morse: Morse--Bott and Degenerate Landscapes}
\label{sec:beyond-morse}
% =========================================================

The Morse hypothesis used in the preceding sections forces the
critical set of the thickness function, and hence the fixed-point set
of the return map, to consist of isolated points. Natural geometric
configurations may instead possess smooth families of critical
points. The first controlled extension is the Morse--Bott setting,
where the Hessian is degenerate precisely along the critical
manifold and remains nondegenerate in the transverse directions.

The exact return-map identities remain available in this setting.
What changes is the spectral structure: tangent directions to a
critical manifold acquire multiplier one, while the transverse Morse
index continues, under strict ascent, to determine the stable
spectral dimension.

We first develop this extension and then indicate what survives when
the critical set is genuinely degenerate.

% ---------------------------------------------------------
\subsection{Morse--Bott critical manifolds and spectral splitting}
\label{subsec:morse-bott-splitting}
% ---------------------------------------------------------

We use the classical Morse--Bott notion introduced by Bott
\cite{bott1954}; see also \cite{nicolaescu2011morse}.

\begin{definition}[Morse--Bott thickness function]
	\label{def:morse-bott-thickness}
	A function
	\[
	d\in C^2(\partial C)
	\]
	is called a \emph{Morse--Bott thickness function} if its critical set
	is a finite disjoint union of compact connected smooth submanifolds
	\begin{equation}
	\operatorname{Crit}(d)
	=
	\bigsqcup_{\alpha\in A}\Sigma_\alpha
	\label{eq:morse-bott-critical-set}
	\end{equation}
	such that, for every $c\in\Sigma_\alpha$,
	\begin{equation}
	\ker H_c
	=
	T_c\Sigma_\alpha,
	\qquad
	H_c=\operatorname{Hess}_{\partial C}d(c).
	\label{eq:morse-bott-kernel}
	\end{equation}
\end{definition}

Thus the only Hessian degeneracy is that forced by motion along the
critical manifold.

Let
\[
r_\alpha
=
\dim\Sigma_\alpha.
\]
Then
\[
\operatorname{rank}H_c
=
m-r_\alpha
\qquad
(c\in\Sigma_\alpha).
\]

\begin{definition}[Morse--Bott index]
	\label{def:morse-bott-index}
	For a connected critical component $\Sigma_\alpha$, define
	\begin{equation}
	\lambda_\alpha
	=
	n_-(H_c),
	\qquad
	c\in\Sigma_\alpha.
	\label{eq:morse-bott-index}
	\end{equation}
\end{definition}

The integer $\lambda_\alpha$ is constant on the connected component:
the transverse Hessian is nondegenerate and its inertia cannot change
without an eigenvalue crossing zero.

Since
\[
\operatorname{Fix}(F)
=
\operatorname{Crit}(d),
\]
every Morse--Bott component is a fixed manifold:
\begin{equation}
F(c)=c
\qquad
(c\in\Sigma_\alpha).
\label{eq:morse-bott-fixed-manifold}
\end{equation}

At every such point the exact linearization remains
\begin{equation}
DF(c)
=
I+B_cH_c.
\label{eq:morse-bott-linearization}
\end{equation}

\begin{proposition}[Multiplier-one eigenspace]
	\label{prop:morse-bott-one-eigenspace}
	For every $c\in\Sigma_\alpha$,
	\begin{equation}
	\boxed{
		\ker(DF(c)-I)
		=
		T_c\Sigma_\alpha.
	}
	\label{eq:one-eigenspace-morse-bott}
	\end{equation}
	In particular, the multiplier $1$ has geometric multiplicity
	$r_\alpha$.
\end{proposition}

\begin{proof}
	Since $B_c$ is invertible,
	\[
	(DF(c)-I)v
	=
	B_cH_cv
	=
	0
	\]
	if and only if
	\[
	H_cv=0.
	\]
	The Morse--Bott condition
	\eqref{eq:morse-bott-kernel} therefore gives
	\[
	\ker(DF(c)-I)
	=
	T_c\Sigma_\alpha.
	\]
\end{proof}

The preconditioned Hessian
\[
K_c
=
B_c^{1/2}H_cB_c^{1/2}
\]
has the same inertia as $H_c$. Hence it possesses

\[
\lambda_\alpha
\quad\text{negative eigenvalues},
\]
\[
m-r_\alpha-\lambda_\alpha
\quad\text{positive eigenvalues},
\]
and
\[
r_\alpha
\quad\text{zero eigenvalues}.
\]

The strict-ascent condition determines the corresponding return
multipliers.

\begin{theorem}[Morse--Bott spectral splitting]
	\label{thm:morse-bott-spectral-splitting}
	Let $\Sigma_\alpha$ be a Morse--Bott critical component and assume
	that, in a neighbourhood of $\Sigma_\alpha$,
	\begin{equation}
	d(F(c))-d(c)
	\ge
	\gamma
	|\nabla_{\partial C}d(c)|^2
	\label{eq:morse-bott-local-ascent}
	\end{equation}
	for some $\gamma>0$.
	
	Then, for every $c\in\Sigma_\alpha$, the derivative admits the
	spectral splitting
	\begin{equation}
	T_c\partial C
	=
	E^s(c)
	\oplus
	T_c\Sigma_\alpha
	\oplus
	E^u(c),
	\label{eq:morse-bott-splitting}
	\end{equation}
	with
	\begin{align}
	\dim E^s(c)
	&=
	\lambda_\alpha,
	\label{eq:morse-bott-stable-dimension}
	\\
	\dim T_c\Sigma_\alpha
	&=
	r_\alpha,
	\label{eq:morse-bott-center-dimension}
	\\
	\dim E^u(c)
	&=
	m-r_\alpha-\lambda_\alpha.
	\label{eq:morse-bott-unstable-dimension}
	\end{align}
	Moreover, the full unit-circle spectral subspace is precisely
	\begin{equation}
	E^c(c)
	=
	T_c\Sigma_\alpha.
	\label{eq:morse-bott-center-space}
	\end{equation}
\end{theorem}

\begin{proof}
	By congruence,
	\[
	K_c
	=
	B_c^{1/2}H_cB_c^{1/2}
	\]
	has the same inertia as $H_c$.
	
	The argument of
	Lemma~\ref{lem:negative-mu-ascent} applies without change to each
	nonzero negative transverse eigenvalue of $K_c$: the corresponding
	Hessian direction is nondegenerate, and the strict-ascent estimate
	\eqref{eq:morse-bott-local-ascent} forces
	\[
	-2<\mu<0.
	\]
	Its return multiplier
	\[
	1+\mu
	\]
	therefore has modulus strictly smaller than one.
	
	Every positive eigenvalue $\mu>0$ produces
	\[
	1+\mu>1,
	\]
	while the $r_\alpha$ zero eigenvalues produce multiplier one.
	
	Because the only zero directions of $H_c$ are
	$T_c\Sigma_\alpha$, Proposition~\ref{prop:morse-bott-one-eigenspace}
	identifies the multiplier-one space with $T_c\Sigma_\alpha$.
	No other multiplier lies on the unit circle under the strict-ascent
	hypothesis. The dimension formulas follow from the inertia count.
\end{proof}

Thus the Morse correspondence extends naturally to
\begin{equation}
\boxed{
	\lambda_\alpha=\dim E^s,
	\qquad
	r_\alpha=\dim E^c,
	\qquad
	m-r_\alpha-\lambda_\alpha=\dim E^u.
}
\label{eq:morse-bott-dictionary}
\end{equation}

% ---------------------------------------------------------
\subsection{Curvature and normal hyperbolicity}
\label{subsec:morse-bott-curvature}
% ---------------------------------------------------------

The curvature representation established in
Section~\ref{sec:morse-spectrum} does not require Morse
nondegeneracy. It holds at every sufficiently regular critical point.

Let
\[
c\in\Sigma_\alpha,
\qquad
d_c=d(c),
\qquad
L_c=I-d_cS_c.
\]
Since
\[
\nabla_{\partial C}d(c)=0,
\]
the core and outer-boundary normals agree at
\[
x=\Phi(c).
\]
Hence
\begin{equation}
\boxed{
	H_c
	=
	L_cS_c^\Omega L_c-L_cS_c,
}
\label{eq:morse-bott-H-curvature}
\end{equation}
and
\begin{equation}
\boxed{
	K_c
	=
	d_c
	\left(
	L_c^{1/2}S_c^\Omega L_c^{1/2}-S_c
	\right).
}
\label{eq:morse-bott-K-curvature}
\end{equation}

In particular, if
\[
v\in T_c\Sigma_\alpha,
\]
then $H_cv=0$, and
\eqref{eq:morse-bott-H-curvature} gives
\begin{equation}
S_c^\Omega L_cv
=
S_cv.
\label{eq:morse-bott-tangent-curvature}
\end{equation}
Thus the neutral tangent directions of the fixed manifold also carry
an exact curvature relation between the two boundaries.

The transverse dynamics is hyperbolic.

\begin{proposition}[Uniform transverse spectral hyperbolicity]
	\label{prop:morse-bott-transverse-hyperbolicity}
	Let $\Sigma_\alpha$ be a compact Morse--Bott component satisfying
	the strict-ascent hypothesis of
	Theorem~\ref{thm:morse-bott-spectral-splitting}. Then all spectral
	directions transverse to $T\Sigma_\alpha$ have multipliers of modulus
	different from one.
	
	Moreover, there exists $\rho\in(0,1)$ such that, uniformly on
	$\Sigma_\alpha$, stable transverse multipliers satisfy
	\[
	|\Lambda|\le\rho,
	\]
	while unstable transverse multipliers satisfy
	\[
	|\Lambda|\ge\rho^{-1}.
	\]
\end{proposition}

\begin{proof}
	Negative transverse eigenvalues of $K_c$ produce multipliers of
	modulus smaller than one, while positive eigenvalues produce
	multipliers larger than one.
	
	Since the transverse Hessian is nondegenerate, no transverse
	eigenvalue of $K_c$ vanishes. Compactness of $\Sigma_\alpha$ and
	continuous dependence of the transverse spectrum then give a uniform
	separation from the unit circle.
\end{proof}

The previous proposition is a spectral statement. To invoke the
standard theory of normally hyperbolic invariant manifolds \cite{hirschPughShub1977}, an
invertible return dynamics must be imposed separately.

\begin{theorem}[Conditional normal hyperbolicity]
	\label{thm:conditional-normal-hyperbolicity}
	Assume, in addition to the hypotheses of
	Proposition~\ref{prop:morse-bott-transverse-hyperbolicity}, that the
	return map is a $C^1$ diffeomorphism on a neighbourhood of
	$\Sigma_\alpha$.
	
	Then $\Sigma_\alpha$ is a normally hyperbolic invariant manifold.
	Its invariant splitting is
	\[
	T_{\Sigma_\alpha}\partial C
	=
	E^s
	\oplus
	T\Sigma_\alpha
	\oplus
	E^u,
	\]
	with
	\begin{equation}
	\operatorname{rank}E^s
	=
	\lambda_\alpha,
	\qquad
	\operatorname{rank}E^u
	=
	m-r_\alpha-\lambda_\alpha.
	\label{eq:morse-bott-bundle-ranks}
	\end{equation}
\end{theorem}

\begin{proof}
	Since
	\[
	F|_{\Sigma_\alpha}=\operatorname{id},
	\]
	one has
	\[
	DF(c)|_{T_c\Sigma_\alpha}=I
	\qquad
	(c\in\Sigma_\alpha).
	\]
	Thus the tangent dynamics along the fixed manifold has unit rate.
	
	By Proposition~\ref{prop:morse-bott-transverse-hyperbolicity},
	the stable, center, and unstable parts of the spectrum of $DF(c)$
	are uniformly separated on the compact manifold $\Sigma_\alpha$.
	Since $DF(c)$ depends continuously on $c$, the corresponding
	spectral projections depend continuously on $c$. They therefore
	define continuous $DF$-invariant vector bundles
	\[
	T_{\Sigma_\alpha}\partial C
	=
	E^s\oplus T\Sigma_\alpha\oplus E^u.
	\]
	The ranks of these bundles are constant and are given by
	\[
	\operatorname{rank}E^s=\lambda_\alpha,
	\qquad
	\operatorname{rank}E^u
	=
	m-r_\alpha-\lambda_\alpha.
	\]
	
	Moreover, Proposition~\ref{prop:morse-bott-transverse-hyperbolicity}
	provides a uniform spectral gap from the unit circle. Since
	$F(c)=c$ for every $c\in\Sigma_\alpha$, the fiber maps
	$DF(c)$ act repeatedly on the same tangent spaces. Compactness and
	the uniform spectral gap therefore yield constants
	\[
	C>0,
	\qquad
	0<\lambda<1,
	\]
	such that
	\[
	\|DF(c)^n v^s\|
	\le
	C\lambda^n\|v^s\|,
	\qquad
	v^s\in E^s_c,
	\]
	and, because $F$ is a local diffeomorphism near $\Sigma_\alpha$,
	\[
	\|DF(c)^{-n}v^u\|
	\le
	C\lambda^n\|v^u\|,
	\qquad
	v^u\in E^u_c,
	\]
	for every $n\ge0$, uniformly in $c\in\Sigma_\alpha$.
	
	After passing, if desired, to an adapted bundle metric, these
	estimates may be taken in the standard one-step form. Since
	\[
	DF(c)|_{T_c\Sigma_\alpha}=I,
	\]
	the transverse contraction and expansion dominate the tangent
	dynamics. These are precisely the normal-hyperbolicity conditions
	for the compact invariant manifold $\Sigma_\alpha$; see
	\cite{hirschPughShub1977}.
\end{proof}

\begin{remark}[Why invertibility is kept separate]
	\label{rem:morse-bott-invertibility}
	The return construction provides a forward map but does not in
	general provide a unique backward return. In particular, a stable
	transverse multiplier may vanish when a preconditioned eigenvalue
	satisfies
	\[
	\mu=-1.
	\]
	The transverse spectral splitting remains meaningful in that case,
	but the standard diffeomorphic theory of normally hyperbolic
	invariant manifolds is not being invoked.
\end{remark}

% ---------------------------------------------------------
\subsection{Morse--Bott inequalities}
\label{subsec:morse-bott-inequalities}
% ---------------------------------------------------------

The topological counting theory also has a Morse--Bott analogue. To
avoid orientability issues for the negative normal bundles, we use
coefficients in
\[
\mathbb F=\mathbb Z_2.
\]

Let
\[
P_{\Sigma_\alpha}(t)
\]
denote the Poincar\'e polynomial of the critical component
$\Sigma_\alpha$.

\begin{definition}[Morse--Bott polynomial]
	\label{def:morse-bott-polynomial}
	The Morse--Bott polynomial of the thickness landscape is
	\begin{equation}
	MB_d(t)
	=
	\sum_{\alpha\in A}
	t^{\lambda_\alpha}
	P_{\Sigma_\alpha}(t).
	\label{eq:morse-bott-polynomial}
	\end{equation}
\end{definition}

The Morse--Bott inequalities give
\begin{equation}
\boxed{
	MB_d(t)
	=
	P_{\partial C}(t)
	+
	(1+t)Q(t),
}
\label{eq:morse-bott-polynomial-relation}
\end{equation}
where $Q$ has nonnegative integer coefficients; see
\cite{bott1954,nicolaescu2011morse}.

Since
\[
\partial C\cong\mathbb S^m,
\]
this becomes
\begin{equation}
\boxed{
	MB_d(t)
	=
	1+t^m+(1+t)Q(t).
}
\label{eq:morse-bott-sphere-relation}
\end{equation}

Thus isolated equilibria are replaced by fixed critical manifolds,
and the contribution of each component records both its transverse
Morse index and its intrinsic topology.

\begin{remark}[No pointwise fixed-point index on a positive-dimensional component]
	\label{rem:morse-bott-fixed-index}
	If
	\[
	r_\alpha>0,
	\]
	then $1$ is a multiplier along $T\Sigma_\alpha$, so
	\[
	\det(I-DF(c))=0
	\qquad
	(c\in\Sigma_\alpha).
	\]
	Consequently, the nondegenerate pointwise fixed-point-index formula
	of Section~\ref{sec:euler-fixed-index} does not apply to individual
	points of a positive-dimensional fixed component.
	
	The Morse--Bott polynomial instead records the topology of the entire
	critical component. A component-wise Lefschetz theory would require
	additional clean-fixed-set machinery and is not developed here.
\end{remark}

\begin{remark}[Relation to recent discrete results]
	\label{rem:morse-bott-endomorphism}
	Morse--Bott inequalities for endomorphisms have recently been studied
	by Mac\'ias-Virg\'os et al.~\cite{maciasVirgos2026}. Their setting is
	different from the geometric return construction considered here and
	their results are not used in the present proofs, but they illustrate
	the broader interaction between critical-manifold topology and
	discrete dynamics.
\end{remark}

% ---------------------------------------------------------
\subsection{An explicit Morse--Bott landscape on $\mathbb S^2$}
\label{subsec:morse-bott-sphere-example}
% ---------------------------------------------------------

Consider the unit spherical core
\[
\partial C=\mathbb S^2
\]
and the positive thickness landscape
\begin{equation}
d(x,y,z)
=
d_0+\varepsilon z^2,
\qquad
d_0>0,
\qquad
\varepsilon>0.
\label{eq:morse-bott-z2}
\end{equation}

Its critical set consists of the equator
\begin{equation}
\Sigma_0
=
\{z=0\}
\cong\mathbb S^1
\label{eq:morse-bott-equator}
\end{equation}
and the two poles
\begin{equation}
p_\pm
=
(0,0,\pm1).
\label{eq:morse-bott-poles}
\end{equation}

Along $\Sigma_0$, the Hessian vanishes in the tangent direction to
the equator and is positive in the transverse meridional direction,
with transverse eigenvalue
\[
2\varepsilon.
\]
Hence
\begin{equation}
r_0=1,
\qquad
\lambda_0=0.
\label{eq:equator-morse-bott-data}
\end{equation}

At each pole,
\[
H_{p_\pm}
=
-2\varepsilon I
\]
on the two-dimensional tangent plane, so each pole is an isolated
maximum of Morse index
\begin{equation}
\operatorname{ind}_M(p_\pm)=2.
\label{eq:morse-bott-pole-index}
\end{equation}

The Morse--Bott polynomial is therefore
\begin{align}
MB_d(t)
&=
P_{\mathbb S^1}(t)
+
2t^2
\nonumber\\
&=
1+t+2t^2.
\label{eq:z2-morse-bott-polynomial}
\end{align}
Since
\[
P_{\mathbb S^2}(t)
=
1+t^2,
\]
we obtain
\begin{equation}
MB_d(t)-P_{\mathbb S^2}(t)
=
t+t^2
=
t(1+t),
\label{eq:z2-morse-bott-correction}
\end{equation}
so that
\begin{equation}
Q(t)=t.
\label{eq:z2-Q}
\end{equation}

The spherical core also makes the transverse return spectrum
explicit. Since
\[
B_c
=
\frac{d(c)}{1+d(c)}I,
\]
the equatorial transverse preconditioned eigenvalue is
\begin{equation}
\mu_{\mathrm{eq}}
=
2\varepsilon
\frac{d_0}{1+d_0}
>0.
\label{eq:equator-mu}
\end{equation}
Hence along the equator
\[
\dim E^s=0,
\qquad
\dim E^c=1,
\qquad
\dim E^u=1,
\]
with transverse multiplier
\begin{equation}
\Lambda_{\mathrm{eq}}
=
1+
2\varepsilon
\frac{d_0}{1+d_0}
>1.
\label{eq:equator-multiplier}
\end{equation}

At either pole,
\[
d_\pm=d_0+\varepsilon,
\]
and both preconditioned eigenvalues are
\begin{equation}
\mu_{\mathrm{pole}}
=
-2\varepsilon
\frac{d_0+\varepsilon}
{1+d_0+\varepsilon}.
\label{eq:pole-morse-bott-mu}
\end{equation}
In the strict-ascent regime,
\[
-2<\mu_{\mathrm{pole}}<0,
\]
so the two poles are attracting fixed points with
\[
\dim E^s(p_\pm)=2.
\]

Thus the dynamical organization is
\[
\boxed{
	\begin{array}{c|c|c|c}
	\text{critical component}
	&
	\dim E^s
	&
	\dim E^c
	&
	\dim E^u
	\\ \hline
	\Sigma_0\cong S^1
	&
	0&1&1
	\\
	p_+
	&
	2&0&0
	\\
	p_-
	&
	2&0&0.
	\end{array}}
\]

The curvature interpretation is equally transparent. Along the
equator, the tangent Hessian direction is zero, so in that direction
the outer-boundary curvature agrees with the corresponding
parallel-surface curvature. The transverse Hessian direction is
positive, so the outer boundary is less curved than the parallel
surface in that direction. At the poles both Hessian directions are
negative, so both outer principal curvatures exceed the corresponding
parallel-surface curvature.

\begin{remark}[Admissibility]
	\label{rem:morse-bott-example-admissibility}
	The calculations above describe the intrinsic Morse--Bott and
	spectral structure of the thickness landscape. The dynamical
	interpretation assumes additionally that the corresponding normal
	graph defines an admissible member of $\mathcal O_C$ and satisfies
	the strict-ascent hypotheses. No claim is made that positivity of
	\eqref{eq:morse-bott-z2} alone guarantees these geometric conditions.
\end{remark}

\paragraph{The parallel configuration.}
The constant-thickness case
\begin{equation}
d(c)\equiv d_0
\label{eq:constant-thickness}
\end{equation}
is the extreme codimension-zero Morse--Bott configuration. Here
\[
\operatorname{Crit}(d)=\partial C,
\qquad
H_c=0,
\qquad
DF(c)=I.
\]
Whenever the corresponding outer parallel surface is admissible, the
exact return map is
\begin{equation}
F=\operatorname{id}_{\partial C}.
\label{eq:parallel-identity-return}
\end{equation}
The entire phase space is then a fixed critical manifold and there
are no transverse directions. Its Morse--Bott polynomial is simply
\[
MB_d(t)=P_{\partial C}(t).
\]

% ---------------------------------------------------------
\subsection{Genuinely degenerate landscapes and surviving structure}
\label{subsec:genuinely-degenerate}
% ---------------------------------------------------------

Morse--Bott degeneracy is controlled: all Hessian degeneracy is
explained by tangency to a smooth critical manifold.

A genuinely degenerate critical point occurs when this structure
fails. For example, if the critical set is a smooth manifold near
$c_*$ but
\begin{equation}
\ker H_{c_*}
\supsetneq
T_{c_*}\operatorname{Crit}(d),
\label{eq:genuine-degeneracy-smooth}
\end{equation}
then there are additional Hessian-null directions not generated by
the critical manifold. A second possibility is that the critical set
itself fails to be a smooth manifold near $c_*$.

In either case, every vector
\[
v\in\ker H_{c_*}
\]
satisfies
\begin{equation}
DF(c_*)v=v.
\label{eq:degenerate-unit-multiplier}
\end{equation}
Thus additional zero Hessian modes produce additional multiplier-one
directions.

The Hessian alone can no longer determine the local dynamics.
Higher-order terms in both the thickness landscape and the exact
return map become essential.

Nevertheless, several basic identities survive without Morse or
Morse--Bott nondegeneracy:
\begin{equation}
F(c)=c
\quad\Longleftrightarrow\quad
\nabla_{\partial C}d(c)=0,
\label{eq:beyond-morse-fixcrit}
\end{equation}
and, at every sufficiently regular critical point,
\begin{equation}
DF(c_*)
=
I+B_*H_*.
\label{eq:beyond-morse-linearization}
\end{equation}
Likewise, the curvature identity
\begin{equation}
H_*
=
L_*S_*^\Omega L_*-L_*S_*
\label{eq:beyond-morse-curvature}
\end{equation}
continues to hold.

If the global ascent estimate remains valid,
\begin{equation}
d(F(c))-d(c)
\ge
\gamma
|\nabla_{\partial C}d(c)|^2,
\label{eq:beyond-morse-ascent}
\end{equation}
then along every orbit
\[
c_{n+1}=F(c_n)
\]
one has
\[
\sum_{n=0}^{\infty}
|\nabla_{\partial C}d(c_n)|^2
<\infty,
\]
because $d(c_n)$ is increasing and bounded above. Consequently,
\begin{equation}
|\nabla_{\partial C}d(c_n)|
\longrightarrow0.
\label{eq:gradient-to-zero}
\end{equation}
Every accumulation point of the orbit therefore belongs to
$\operatorname{Crit}(d)$.

Strict ascent also continues to exclude nontrivial periodic and
noncritical recurrent motion. What is no longer guaranteed is
convergence of every orbit to a unique isolated equilibrium.

The hierarchy may therefore be summarized as
\[
\boxed{
	\begin{array}{c|c|c}
	\text{landscape}
	&
	\text{critical set}
	&
	\text{linearized return structure}
	\\ \hline
	\text{Morse}
	&
	\text{isolated points}
	&
	E^s\oplus E^u
	\\
	\text{Morse--Bott}
	&
	\text{smooth critical manifolds}
	&
	E^s\oplus T\Sigma\oplus E^u
	\\
	\text{genuinely degenerate}
	&
	\text{possibly singular}
	&
	\text{additional center directions}.
	\end{array}}
\]

\begin{remark}[Limits of the present theory]
	\label{rem:degenerate-limits}
	At a genuinely degenerate critical point:
	
	\begin{itemize}
		\item
		the Morse normal form is unavailable;
		
		\item
		the Morse-index/stable-dimension correspondence need not have a
		direct analogue;
		
		\item
		ordinary Morse handle attachment must be replaced by more general
		critical-point theory;
		
		\item
		the nondegenerate pointwise fixed-point-index formula may fail,
		especially when the fixed point is nonisolated;
		
		\item
		center-manifold or higher-order normal-form analysis may be required
		to determine the local return dynamics.
	\end{itemize}
	
	Generalized forms of Morse theory can retain substantial topological
	information beyond classical nondegeneracy; see, for example,
	Kirwan and Penington~\cite{kirwanPenington2021}. Adapting such tools
	to thickness landscapes constrained by $\mathcal O_C$ remains a
	separate problem.
\end{remark}
% =========================================================
% =========================================================
\section{Discussion and Open Problems}
\label{sec:discussion-open-problems}
% =========================================================

The preceding sections show that the return map in the class
$\mathcal O_C$ carries a structure considerably richer than that of
an arbitrary discrete dynamical system on a sphere.  Its equilibria,
linearized spectrum, fixed-point indices, invariant manifolds, and
level-set topology are all generated by the same scalar thickness
landscape together with the curvature of the two boundaries.

The purpose of this section is not to repeat the results already
proved, but to isolate the structural picture that emerges from them
and to identify the questions that remain genuinely open.

% ---------------------------------------------------------
\subsection{The geometry--dynamics--topology correspondence}
\label{subsec:discussion-correspondence}
% ---------------------------------------------------------

At a critical point $c_*$ of the thickness function, write
\[
d_*=d(c_*),
\qquad
L_*=I-d_*S_*,
\]
\[
H_*=\operatorname{Hess}_{\partial C}d(c_*),
\qquad
B_*=d_*L_*^{-1},
\]
and
\[
K_*=B_*^{1/2}H_*B_*^{1/2}.
\]

The return dynamics is governed by the exact linearization
\begin{equation}
DF(c_*)
=
I+B_*H_*.
\label{eq:discussion-linearization}
\end{equation}
At the same time, the geometry of the two boundaries gives
\begin{equation}
H_*
=
L_*S_*^\Omega L_*-L_*S_*,
\label{eq:discussion-H-curvature}
\end{equation}
and therefore
\begin{equation}
\boxed{
	K_*
	=
	d_*
	\left(
	L_*^{1/2}S_*^\Omega L_*^{1/2}-S_*
	\right).
}
\label{eq:discussion-K-curvature}
\end{equation}

Thus the chain from geometry to dynamics is exact:
\begin{equation}
\boxed{
	(S_*,S_*^\Omega,d_*)
	\longrightarrow
	H_*
	\longrightarrow
	K_*
	\longrightarrow
	DF(c_*).
}
\label{eq:geometry-to-dynamics-chain}
\end{equation}

When $c_*$ is Morse, positive congruence gives
\[
\operatorname{inertia}(K_*)
=
\operatorname{inertia}(H_*).
\]
Under strict ascent,
\[
\dim E^s(c_*)
=
\operatorname{ind}_M(c_*),
\qquad
\dim E^u(c_*)
=
m-\operatorname{ind}_M(c_*).
\]

Topology then enters at two levels. Globally,
\[
\partial C\cong\mathbb S^m,
\]
so the Morse polynomial satisfies
\[
M_d(t)
=
1+t^m+(1+t)Q(t),
\]
and therefore
\[
\#\operatorname{Crit}(d)
=
2+2Q(1).
\]
In particular, every Morse return map in the present setting has an
even number of equilibria.

Locally, the same Morse index controls the topology of the thickness
filtration:
\[
\text{sublevel handle index}
=
\dim E^s(c_*),
\]
\[
\text{superlevel handle index}
=
\dim E^u(c_*).
\]
Moreover,
\[
\operatorname{ind}_{\mathrm{fix}}(F,c_*)
=
(-1)^{\dim E^u(c_*)},
\]
and this sign is precisely the Euler-characteristic jump of the
superlevel filtration across an isolated critical level.

The resulting correspondence may therefore be summarized as
\begin{equation}
\boxed{
	\begin{array}{ccccc}
	\text{boundary curvature}
	&
	\longrightarrow
	&
	\text{thickness Hessian}
	&
	\longrightarrow
	&
	\text{return spectrum}
	\\[1mm]
	&&
	\operatorname{ind}_M
	&
	\longleftrightarrow
	&
	\dim E^s
	\\[1mm]
	&&
	\updownarrow
	&&
	\updownarrow
	\\[-1mm]
	&&
	\text{Morse filtration}
	&
	\longleftrightarrow
	&
	\text{fixed-point index}.
	\end{array}}
\label{eq:full-correspondence}
\end{equation}

This correspondence is the main structural conclusion of the paper.

% ---------------------------------------------------------
\subsection{Three distinct spectral thresholds}
\label{subsec:three-thresholds}
% ---------------------------------------------------------

The preceding analysis reveals three different spectral values that
should not be conflated.

Let
\[
\mu\in\sigma(K_*),
\qquad
\Lambda=1+\mu
\]
be the associated return multiplier.

The first threshold is
\begin{equation}
\mu=0,
\qquad
\Lambda=1.
\label{eq:threshold-zero}
\end{equation}
Here the Hessian becomes degenerate. At an isolated critical point
this is the boundary of Morse nondegeneracy; in a Morse--Bott
configuration it occurs precisely in directions tangent to the
critical manifold.

The second threshold is
\begin{equation}
\mu=-1,
\qquad
\Lambda=0.
\label{eq:threshold-minus-one}
\end{equation}
This does not destroy fixed-point nondegeneracy, since
\[
1\notin\sigma(DF(c_*))
\]
may still hold. Instead, it is the threshold at which
\[
DF(c_*)
\]
loses invertibility. It is therefore relevant to the existence of
local inverse dynamics, unstable manifolds in the diffeomorphic
sense, and eventually to global invertibility.

The third threshold is
\begin{equation}
\mu=-2,
\qquad
\Lambda=-1.
\label{eq:threshold-minus-two}
\end{equation}
Here a negative Morse direction reaches the boundary of linear
stability. The Morse index need not change, because the Hessian
eigenvalue has not crossed zero, but the dynamical classification of
that direction may change.

Thus the three thresholds represent genuinely different phenomena:
\begin{equation}
\boxed{
	\begin{array}{c|c|c}
	\mu & \Lambda & \text{phenomenon}\\ \hline
	0 & 1 & \text{loss of Morse nondegeneracy}\\
	-1 & 0 & \text{loss of local invertibility}\\
	-2 & -1 & \text{loss of stable Morse dynamics}.
	\end{array}}
\label{eq:three-threshold-table}
\end{equation}

The present theory is primarily concerned with the Morse,
strict-ascent regime. The thresholds $\mu=-1$ and $\mu=-2$ identify
two distinct directions in which the global dynamical theory may
require additional analysis.

% ---------------------------------------------------------
\subsection{Degree one and the global invertibility problem}
\label{subsec:degree-global-invertibility}
% ---------------------------------------------------------

Section~\ref{sec:euler-fixed-index} shows that, under the Morse
hypotheses,
\begin{equation}
\deg(F)=1.
\label{eq:discussion-degree-one}
\end{equation}
This does not by itself imply that $F$ is invertible, but it turns
global invertibility into a purely local-rank question.

\begin{proposition}[From local to global invertibility]
	\label{prop:local-to-global-invertibility}
	Assume that the return map satisfies the hypotheses leading to
	\eqref{eq:discussion-degree-one} and that
	\begin{equation}
	\det DF(c)\neq0
	\qquad
	\text{for every }c\in\partial C.
	\label{eq:global-nonsingularity}
	\end{equation}
	Then
	\[
	F:\partial C\longrightarrow\partial C
	\]
	is a global $C^1$ diffeomorphism.
\end{proposition}

\begin{proof}
	Condition~\eqref{eq:global-nonsingularity} makes $F$ a local
	diffeomorphism. Since $\partial C$ is compact and connected, $F(\partial C)$ is open because \(F\) is a local diffeomorphism, and compact hence closed because the domain is compact. Since the target \(\partial C\) is connected, the image is all of \(\partial C\). A proper surjective local diffeomorphism is then a  finite-sheeted covering map.
	
	For a covering of an oriented closed connected manifold, the
	absolute value of the degree equals the number of sheets. Since
	\[
	\deg(F)=1,
	\]
	the covering has one sheet. Hence $F$ is bijective and therefore a
	global diffeomorphism.
\end{proof}

At an equilibrium the local condition is explicit:
\[
DF(c_*)\ \text{invertible}
\quad\Longleftrightarrow\quad
-1\notin\sigma(K_*).
\]
The unresolved part of the problem lies away from the equilibrium
set, where no comparable intrinsic criterion is currently available.

% ---------------------------------------------------------
\subsection{Open problems}
\label{subsec:open-problems}
% ---------------------------------------------------------

The results of the paper reduce several previously broad questions to
more precise geometric and dynamical problems.

\begin{openproblem}[Genericity in the full admissible class]
	\label{op:full-admissible-genericity}
	Section~3 establishes Morse genericity inside admissible
	finite-dimensional parameter families satisfying the stated spanning
	condition.
	
	Determine whether an analogous genericity statement holds in the
	full geometrically admissible class generated by domains
	$\Omega\in\mathcal O_C$.
	
	More precisely, identify a natural topology on admissible thickness
	landscapes and determine whether Morse thickness functions form an
	open dense or residual subset.
	
	The essential difficulty is that perturbations must preserve
	simultaneously the geometric normal property, radial accessibility,
	regularity, and return transversality.
\end{openproblem}

\begin{openproblem}[Realizability of prescribed thickness landscapes]
	\label{op:thickness-realizability}
	Characterize the positive functions
	\[
	d:\partial C\longrightarrow(0,\infty)
	\]
	for which
	\[
	\Phi(c)=c+d(c)\nu(c)
	\]
	is the outer boundary of a domain
	\[
	\Omega\in\mathcal O_C.
	\]
	
	Positivity and regularity of $d$ are not sufficient. The resulting
	normal graph must satisfy the geometric intersection and
	transversality conditions required for the exact return construction.
	
	A satisfactory answer would characterize the image of the map
	\[
	\Omega
	\longmapsto
	d_\Omega
	\]
	and would separate intrinsic restrictions on the thickness landscape
	from restrictions imposed by the return geometry.
\end{openproblem}

\begin{openproblem}[Global invertibility of the return map]
	\label{op:global-invertibility}
	Find geometric conditions on $(C,\Omega)$ guaranteeing
	\[
	\det DF(c)\neq0
	\qquad
	(c\in\partial C).
	\]
	
	By Proposition~\ref{prop:local-to-global-invertibility}, any such
	criterion would imply that the degree-one return map is a global
	diffeomorphism.
	
	At equilibria the obstruction is completely characterized by
	\[
	-1\in\sigma(K_*).
	\]
	The principal unresolved issue is therefore to understand critical
	points of the map $F$ itself away from
	\[
	\operatorname{Crit}(d).
	\]
\end{openproblem}

\begin{openproblem}[Sharp curvature criteria in the noncommuting case]
	\label{op:noncommuting-curvature}
	At every critical point the spectrum is governed exactly by
	\[
	K_*
	=
	d_*
	\left(
	L_*^{1/2}S_*^\Omega L_*^{1/2}-S_*
	\right).
	\]
	Thus an invariant curvature representation is already available
	without alignment of principal directions.
	
	The remaining problem is to extract from this operator identity sharp
	geometric criteria for its spectrum when
	\[
	[S_*,S_*^\Omega]\neq0.
	\]
	
	In particular, seek conditions expressed directly through principal
	curvatures, relative principal-frame geometry, generalized
	eigenvalues, or operator inequalities that determine:
	
	\[
	\operatorname{inertia}(K_*),
	\qquad
	-1\in\sigma(K_*),
	\qquad
	-2\in\sigma(K_*),
	\]
	or quantitative distance from these thresholds.
	
	Such criteria would turn the exact operator formula into directly
	checkable geometric conditions for Morse type, local invertibility,
	and stability.
\end{openproblem}

\begin{openproblem}[Invariant manifolds and Morse handles]
	\label{op:manifolds-versus-handles}
	The present theory proves the dimension identities
	\[
	\dim W^s_{\mathrm{loc}}(c_*)
	=
	\text{sublevel handle index},
	\]
	and, under local invertibility,
	\[
	\dim W^u_{\mathrm{loc}}(c_*)
	=
	\text{superlevel handle index}.
	\]
	
	Determine additional conditions under which these invariant
	manifolds realize the actual cells or attaching data of the Morse
	filtration.
	
	Under suitable global invertibility and transversality assumptions,
	one may ask whether the return map admits a geometrically generated
	analogue of the Morse--Smale complex and whether its incidence
	relations can be recovered from the thickness landscape.
\end{openproblem}

\begin{openproblem}[Morse--Smale genericity and basin geometry]
	\label{op:morse-smale-genericity}
	Within a suitable diffeomorphic subfamily of $\mathcal O_C$,
	determine whether transverse intersection of stable and unstable
	manifolds is generic.
	
	If so, the heteroclinic graph is constrained by the two orderings
	already proved:
	\begin{equation}
	p\longrightarrow q
	\quad\Longrightarrow\quad
	d(p)<d(q),
	\qquad
	\operatorname{ind}_M(p)
	\le
	\operatorname{ind}_M(q).
	\label{eq:open-heteroclinic-order}
	\end{equation}
	
	A related problem is to characterize basin boundaries. For an
	attracting Morse maximum $a$, the present theory gives
	\[
	\partial_{\partial C}\mathcal B(a)
	\subset
	\bigcup_{q\ {\rm nonattracting}}
	W^s(q).
	\]
	Determine conditions under which this inclusion is an equality and
	under which the boundary is a regular union of stable manifolds of
	saddle equilibria.
\end{openproblem}

\begin{openproblem}[Bifurcation at the $\mu=-2$ threshold]
	\label{op:minus-two-bifurcation}
	Consider a one-parameter admissible family for which an eigenvalue of
	$K_*$ crosses
	\[
	\mu=-2.
	\]
	At the threshold the corresponding return multiplier is
	\[
	\Lambda=-1.
	\]
	
	Classify the local bifurcations of the exact return map compatible
	with the geometric constraints of $\mathcal O_C$.
	
	In one dimension a period-doubling bifurcation is a natural
	possibility, but the multiplier condition alone is insufficient.
	The relevant higher-order normal-form coefficients, transversality
	with respect to the parameter, and persistence of admissibility must
	also be established.
	
	Simultaneous crossings of several eigenvalues raise a
	higher-dimensional center-dynamics problem.
\end{openproblem}

\begin{openproblem}[Global dynamics beyond Morse]
	\label{op:beyond-morse-global}
	Section~\ref{sec:beyond-morse} establishes the local spectral
	splitting of Morse--Bott critical manifolds and, under suitable
	invertibility assumptions, their normal hyperbolicity.
	
	Develop the corresponding global theory.
	
	For Morse--Bott landscapes, questions include convergence to a
	single point or component, the structure of global stable
	foliations, interactions among normally hyperbolic fixed manifolds,
	and a component-wise analogue of the filtration and fixed-point-index
	picture.
	
	For genuinely degenerate landscapes, determine which additional
	hypotheses guarantee convergence when multiplier-one directions are
	not explained by a smooth critical manifold. Higher-order return-map
	expansions, center-manifold methods, generalized Morse theory, or
	analytic-gradient inequalities may be required.
\end{openproblem}

\begin{openproblem}[Robustness of the admissible ascent regime]
	\label{op:ascent-robustness}
	Quantify the stability of the complete regime considered in this
	paper under perturbations of the convex core and outer boundary.
	
	In particular, find quantitative conditions guaranteeing persistence
	of the simultaneous properties
	\[
	\Omega\in\mathcal O_C,
	\qquad
	d\ \text{Morse},
	\]
	and
	\[
	d(F(c))-d(c)>0
	\qquad
	(c\notin\operatorname{Crit}(d)).
	\]
	
	A useful theory should control not only admissibility and ascent but
	also spectral margins from the three distinguished thresholds
	\[
	\mu=0,\qquad -1,\qquad -2.
	\]
	Such estimates would provide a quantitative notion of robustness for
	both the geometry and the induced return dynamics.
\end{openproblem}

% ---------------------------------------------------------
\subsection{Directions of continuation}
\label{subsec:continuation-directions}
% ---------------------------------------------------------

The open problems above fall naturally into three complementary
directions.

The first concerns the geometry of admissible thickness landscapes:
\[
\boxed{
	\text{realizability}
	\longrightarrow
	\text{genericity}
	\longrightarrow
	\text{robustness}.
}
\]

The second concerns global organization of the return dynamics:
\[
\boxed{
	\text{local invertibility}
	\longrightarrow
	\text{global diffeomorphism}
	\longrightarrow
	\text{Morse--Smale structure}.
}
\]

The third concerns departures from the nondegenerate strict-ascent
regime:
\[
\boxed{
	\begin{array}{c}
	\mu=-2
	\longrightarrow
	\text{bifurcation and recurrent dynamics},
	\\[1mm]
	\mu=0
	\longrightarrow
	\text{Morse--Bott or degenerate critical structure}.
	\end{array}}
\]

The threshold $\mu=-1$ intersects the second direction by marking the
loss of local invertibility.

These directions are not independent.  The admissible geometry
determines which thickness landscapes and parameter families can
occur; curvature determines their local return spectrum; the
spectrum controls invariant-manifold dimensions and stability; and
the thickness filtration constrains the resulting global
organization.

The broader return-map program can therefore be viewed as the study
of the chain
\begin{equation}
\boxed{
	\begin{aligned}
	&\text{convex geometry}
	\longrightarrow
	\text{admissible thickness landscape}
	\longrightarrow
	\text{return dynamics}\\
	&\quad\longrightarrow
	\text{topological organization}.
	\end{aligned}
}
\label{eq:return-program-chain}
\end{equation}
% =========================================================
% =========================================================
\section{Conclusion}
\label{sec:conclusion}
% =========================================================

We have developed a Morse-theoretic description of the thickness
landscape associated with the geometric return map in the class
$\mathcal O_C$.  The central point is that the topology of the
thickness function, the curvature of the two boundaries, and the
linearized return dynamics are not independent structures.

At a critical point $c_*$ of the thickness function, the exact
linearization is
\[
DF(c_*)
=
I+B_*H_*,
\]
where
\[
H_*
=
\operatorname{Hess}_{\partial C}d(c_*),
\qquad
B_*
=
d_*(I-d_*S_*)^{-1}.
\]
The positive definiteness of $B_*$ implies that the
curvature-preconditioned Hessian
\[
K_*
=
B_*^{1/2}H_*B_*^{1/2}
\]
has the same inertia as $H_*$.  Moreover, the geometry of the two
boundaries yields the exact curvature representation
\[
K_*
=
d_*
\left(
L_*^{1/2}S_*^\Omega L_*^{1/2}-S_*
\right),
\qquad
L_*=I-d_*S_*.
\]
Thus the spectrum of the return map is determined directly by the
interaction between the thickness landscape and the curvature of the
core and outer boundary.

For Morse thickness functions, this relation becomes dynamical in the
strict-ascent regime. Every negative eigenvalue $\mu$ of $K_*$ lies
in
\[
-2<\mu<0,
\]
and therefore
\[
\dim E^s(c_*)
=
\operatorname{ind}_M(c_*),
\qquad
\dim E^u(c_*)
=
m-\operatorname{ind}_M(c_*).
\]
Hence minima of the thickness landscape are repelling equilibria,
maxima are attracting equilibria, and intermediate Morse indices give
saddle dynamics.  Under aligned principal directions, the same Morse
index can be read geometrically as the number of directions in which
the outer boundary is more curved than the corresponding parallel
surface of the convex core.

Convexity imposes a second, global layer of structure. Since
\[
\partial C\cong\mathbb S^m,
\]
the Morse polynomial satisfies
\[
M_d(t)
=
1+t^m+(1+t)Q(t),
\]
and consequently
\[
\#\operatorname{Crit}(d)
=
2+2Q(1).
\]
Thus every Morse return map in the present setting has an even number
of equilibria. In the strict-ascent regime the same polynomial
describes their distribution by stable spectral dimension.

The local fixed-point index provides a further bridge between the
discrete dynamics and topology:
\[
\operatorname{ind}_{\mathrm{fix}}(F,c_*)
=
(-1)^{m+\operatorname{ind}_M(c_*)}
=
(-1)^{\dim E^u(c_*)}.
\]
Summing these local indices recovers the Euler characteristic of
$\partial C$, and the Lefschetz--Hopf formula gives
\[
\deg(F)=1.
\]
The degree-one property, together with local nonsingularity of the
return map, reduces global invertibility to a covering-space
argument.

The Morse filtration gives a complementary global interpretation.
At a critical point of index $\lambda$, the sublevel handle has index
$\lambda$ and the superlevel handle has index $m-\lambda$.  Under
strict ascent these are precisely the stable and unstable spectral
dimensions.  Since the superlevel sets
\[
U^a=\{d\ge a\}
\]
are forward invariant, their filtration is dynamically oriented, and
the Euler-characteristic jump across an isolated critical level is
exactly the local fixed-point index of the corresponding equilibrium.

The invariant-manifold analysis shows where additional dynamical
hypotheses enter.  Stable manifolds follow from spectral
hyperbolicity, whereas unstable manifolds and a genuine
Morse--Smale organization require suitable invertibility and
transversality assumptions.  Under these assumptions, heteroclinic
connections are constrained simultaneously by increasing thickness
and nondecreasing Morse index.

Finally, the framework extends in a controlled way beyond isolated
Morse critical points.  For Morse--Bott landscapes, fixed points are
replaced by fixed critical manifolds and the linearized splitting
becomes
\[
E^s\oplus T\Sigma\oplus E^u.
\]
The transverse Morse--Bott index continues to determine the stable
spectral dimension, while the topology of the critical manifold
enters the Morse--Bott polynomial.  Genuinely degenerate landscapes,
by contrast, introduce additional multiplier-one directions and
require higher-order analysis.

The resulting picture is therefore one in which a single scalar
quantity---the thickness function---organizes several levels of the
return problem:
\[
\boxed{
\begin{aligned}
	&\text{convex geometry}
	\longrightarrow
	\text{thickness landscape}
	\longrightarrow
	\text{return spectrum and dynamics}\\
	&\quad\longrightarrow
	\text{Morse and fixed-point topology}.
\end{aligned}
}
\]
The theory developed here describes the nondegenerate
strict-ascent layer of this correspondence and identifies the
geometric, spectral, and topological mechanisms that must be
understood to move beyond it.

\section*{Acknowledgments}
% ===============================

% ===============================

\end{document}